\documentclass[a4paper,11pt]{article}

\usepackage[english]{babel}
\usepackage[T1]{fontenc}
\usepackage{amssymb,amsfonts}
\usepackage{mathtools}
\usepackage{centernot}
\usepackage{esint}
\usepackage{amsmath,amsthm}
\usepackage{graphicx}
\usepackage{fancyhdr}
\usepackage{upgreek}
\usepackage{xcolor}
\usepackage{enumitem}
\setlist[enumerate,1]{label=(\arabic*)}
\setlist[enumerate,2]{label=(\alph)}
\usepackage[toc,page]{appendix}
\usepackage{thmtools,thm-restate}
\usepackage{xfrac}
\usepackage{csquotes}
\usepackage[numbers]{natbib}
\usepackage{notes2bib}
\usepackage{hyperref}
\hypersetup{
  pdfauthor={},
  pdftitle={},
  pdfsubject={},
  urlcolor=blue,
}

\newtheorem{proposition}{Proposition}[section]
\newtheorem{lemma}[proposition]{Lemma}
\newtheorem{remark}[proposition]{Remark}
\newtheorem*{remark_nn}{Remark}
\newtheorem{theorem}[proposition]{Theorem}
\newtheorem{corollary}[proposition]{Corollary}
\newtheorem{definition}[proposition]{Definition}
\newtheorem*{lemma*}{Lemma}
\newtheorem{assumption}[proposition]{Assumption}

\newcommand{\ignore}[1]{}

\newcommand{\weakto}{\rightharpoonup}

\newcommand{\Chi}{\mathcal{X}}
\newcommand{\Ha}{\mathcal{H}}
\newcommand{\LL}{\mathcal{L}}

\newcommand{\eps}{\varepsilon}

\newcommand{\R}{\mathbb{R}}
\newcommand{\N}{\mathbb{N}}

\newcommand{\ssubset}{\subset\mathrel{\mkern-3mu}\subset}
\newcommand{\cmp}[1]{{#1}^{\mathsf{c}}}
\newcommand{\calH}{\mathcal{H}}
\newcommand{\interior}[1]{%
{\kern0pt#1}^{\mathrm{o}}%
}

\DeclareMathOperator{\dist}{dist}

\newcommand{\Hdist}{\dist_{\Ha}}

\numberwithin{equation}{section}

\definecolor{carminered}{rgb}{1.0, 0.0, 0.22}
\definecolor{islamicgreen}{rgb}{0.0, 0.56, 0.0}
\definecolor{blue(ryb)}{rgb}{0.01, 0.28, 1.0}

\begin{document}

\title{Qualitative properties of solutions to a non-local free boundary problem modeling cell polarization}
\author{Anna~Logioti\footnotemark[1] \quad Barbara~Niethammer\thanks{Institute for Applied Mathematics,
University of Bonn} \quad  Matthias~R\"oger\thanks{Mathematics faculty,
Technische Universit\"at Dortmund}  \quad Juan~J.~L.~Vel\'azquez\footnotemark[1]}
\maketitle

\begin{abstract}
We consider a parabolic non-local free boundary problem that has been derived as a limit of a bulk-surface reaction-diffusion system which models cell polarization. The authors have justified the well-posedness of this problem and have further proved uniqueness of solutions and global stability of steady states. In this paper we investigate qualitative properties of the free boundary.
We present necessary and sufficient conditions for the initial data that imply continuity of the support at $t=0$. If one of these assumptions fail, then jumps of the support take place. In addition we provide a complete characterization of the jumps for a large class of initial data.

\medskip

{\bf Keywords. } non-local free boundary problem, obstacle problem, estimates on the support of solutions

\end{abstract}

\tableofcontents

\section{Introduction} \label{Intro}
Obstacle problems appear in various applications and are still an active field of current research. In this work we consider a particular parabolic obstacle problem on a surface in $\R^3$ that is motivated by a model for cell polarization and that enjoys a mass conservation property.
This property leads to some particular features and challenges of the model.

\medskip
The classical obstacle problem can be formulated as
\begin{align}
    &u\geq 0, \label{eq:obs-class1}\\
    &\partial_tu-\Delta u \geq f, \label{eq:obs-class2}\\
    &\partial_tu-\Delta u = f\quad\text{ in }\{u>0\}, \label{eq:obs-class3}
\end{align}
see for example \cite[Section 3.1]{RosO18}.
One of the distinct features of this kind of problems is the possibility of compactly supported solutions.
If $f$ is strictly negative it is well known that solutions remain compactly supported if the initial data have this property \cite{BF76,EK79}.
The characterization of the coincidence set $\{u=0\}$ and its (moving) free boundary is then crucial for the understanding of the problem.

\medskip
In \cite{BF76,EK79} it is shown that under suitable assumptions on the initial data the support of $u(\cdot,t)$ has distance at most of order $\sqrt{t}$ from the support of $u_0$.
Fine regularity properties of the (moving) free boundary have been obtained in the context of the classical one-phase Stefan problem in seminal work by Caffarelli \cite{Caff77}, see also the exposition in \cite[Section 2.9]{F82}, and have later been extended to the general case \cite{CPS04}.
We refer to \cite{CaFi13}, \cite{Fig18}, \cite{FiSe19}, \cite{RosO18} and the references therein for more recent developments and extensions to more general operators.

\bigskip
We are concerned with a parabolic obstacle problem that was obtained as an asymptotic reduction for a cell polarization model \cite{LNRV21} (see also \cite{NRV20} for the stationary case).

The spatial domain is now given by a two-dimensional manifold $\Gamma\subset\R^3$ without boundary.
For given $g:\Gamma\to (0,1)$ and initial data $u_0:\Gamma\to [0,\infty)$ we consider the system
\begin{align}
  &u \geq 0\;, \label{eq:oblam1}
  \\
  &\partial_t u -\Delta_\Gamma u \geq -1+\frac{g}{\lambda(t)}\;, \label{eq:oblam2}
  \\
  &\partial_t u -\Delta_\Gamma u = -1+\frac{g}{\lambda(t)}
  \quad\text{ in }\{u>0\}\;, \label{eq:oblam3}\\
  & \lambda(t) = \fint_{\{u(\cdot,t)>0\}}g\;,
  \label{eq:oblam4}
\end{align}
where $\Delta_\Gamma$ denotes the Laplace-Beltrami operator on $\Gamma$.
Due to \eqref{eq:oblam4} this system has the property that the total mass of $u$ is conserved, i.e.
\begin{equation}
  t\mapsto \int_\Gamma u(\cdot,t)\,dS \quad\text{ is constant. }
  \label{1eq:massconservation}
\end{equation}
The function $\lambda$ can be understood as the associated Lagrange multiplier.

The mass conservation property renders the problem \eqref{eq:oblam1}-\eqref{eq:oblam4} nonlocal, with a right hand side that depends on the {\em support} of the solution.
As we will see, this feature leads to a rather intricate behavior and makes the analysis quite challenging.

\medskip
Mass conservation is a very natural property in many applications. To the best of our knowledge it has hardly been considered in the context of obstacle problems.
One exception is \cite{AmNR17}, where a nonlinear conservation law is coupled to an obstacle condition and the constraint of mass conservation, which is much more difficult from a well-posedness point of view than a parabolic obstacle problem.
The authors prove existence of an entropy solution, but neither uniqueness of solutions nor regularity issues are considered.
In the context of parabolic obstacle problems of fourth order, mass conservation constraints have been considered for Cahn-Hilliard type problems with double obstacle potential or for a linear biharmonic heat equation with obstacle condition.
For the first application see for example \cite{BlEl91}, where existence and regularity of solutions have been discussed. The regularity of the boundary of the coincidence sets has not been considered.
The second application appears in a certain thin film limit, see \cite{BeHQ00}.
The analysis is restricted to self-similar structures and does not study general continuity properties of the moving boundary.

In view of the relevance of a mass conservation property and the apparent novelty of studying regularity properties in such a context, our analysis may be prototypical for a larger class of applications and opens interesting perspectives for future research in this direction.

\medskip
The model \eqref{eq:oblam1}-\eqref{eq:oblam4} has been derived in \cite{NRV20}, \cite{LNRV21} taking as a starting point a reaction-diffusion system that describes the concentrations of some chemicals inside the cell and on the membrane.
In particular, the variable $u$ in \eqref{eq:oblam1}-\eqref{eq:oblam4} represents the concentration of an activated protein on the cell membrane $\Gamma$ and the particular form of the Lagrange multiplier guaranteeing mass conservation is induced by the volume-surface coupling in the original reaction-diffusion system.

We adopt here the setting of the cell polarization model. On the other hand, it would also be meaningful to consider \eqref{eq:oblam1}-\eqref{eq:oblam4} in a domain $\Omega\subset\R^n$ subject to suitable boundary conditions.
We also mention two settings which lead to models analogue to \eqref{eq:oblam1}-\eqref{eq:oblam4}.
This is first the study of obstacle problems for which the amount of matter contained below an elastic membrane is fixed, and second --closer to our derivation-- a chemical system modeled by a reaction-diffusion system in which the diffusion coefficient associated to one of the species is much larger than the other (compare the shadow system reduction that is very common for two-variable reaction-diffusion systems in open domains and that has also been considered in the context of obstacle problems in \cite{Rodr02}).

\medskip
Whereas existence of solutions and convergence towards a unique stationary state have been discussed in \cite{LNRV21}, our focus here is on the characterization of qualitative properties of solutions, namely continuity properties of the Lagrange multiplier and of the (compact) support of the solutions.

We find that there are two conditions on the initial data $u_0\in C^0(\Gamma)$, $u_0\geq 0$ that play a crucial role in proving either continuity or jumps of the set $\{u(\cdot,t)>0\}$ as $t \to 0^{+}$.
The first nondegeneracy condition assumes
% that for some fixed $\theta>0$ it holds
\begin{equation}\label{eq:nondegeneracy1-lambda}
 1-\frac{g}{\lambda_0} >0 \quad\text{ in } \{u_0=0\} \;,
\end{equation}
where $\lambda_0:= \fint_{\{u_0>0\}} g\,dS$.
We impose a second nondegeneracy condition in form of a mild regularity assumption
\begin{equation}\label{eq:nondegeneracy2}
  \calH^2\big(\partial \{u_0>0\}\big) =0\;.
\end{equation}

Our first main result concerns the right-continuity of the support at time $t=0$ and can be stated as follows (see Section \ref{Continuity Section} for a precise formulation).

\begin{theorem}\label{1thm:main1}
Suppose that \eqref{eq:nondegeneracy1-lambda} and \eqref{eq:nondegeneracy2} hold true.
Then
\begin{equation*}
	\lim_{t\searrow 0}\lambda(t)=\lambda_0 \quad\text{ and }\quad
  \lim_{t\searrow 0}\{u(\cdot,t)>0 \}= \{u_0>0\}\;,
\end{equation*}
where the convergence of supports holds both in the sense of $L^1(\Gamma)$-convergence of the associated characteristic functions and in the sense of Hausdorff convergence.
\end{theorem}

Let us point out that {\em a priori} we cannot apply the growth bounds from \cite{EK79} for solutions of the classical obstacle problem \eqref{eq:obs-class1}-\eqref{eq:obs-class3}.
Under just a boundedness and global nondegeneracy condition on the right-hand side $f$ it is proved in \cite[Theorem 3.2]{EK79} that the distance of the support of $u(\cdot,t)$ to the support of the initial data grows at most with rate $\sqrt{t|\log t|}$.
However, by the peculiar dependence of the right-hand side in \eqref{eq:oblam1}-\eqref{eq:oblam4} on the support of $u$ the nondegeneracy condition \eqref{eq:nondegeneracy1-lambda} does not necessarily propagate to positive times.
To show such a propagation property of the nondegeneracy condition is therefore a key step in our proof of Theorem \ref{1thm:main1}.
{\em A posteriori} \cite[Theorem 3.2]{EK79} can eventually be applied and even yields a growth rate of the support, see Corollary \ref{cor:growthbound} below.

\medskip
Our second main theorem shows that the support may experience an initial jump if \eqref{eq:nondegeneracy1-lambda} is violated.
% Note that the latter condition is equivalent to $g<\lambda_0$ in $\{u_0=0\}$.
In order to avoid too many technicalities, in this introduction we give a slightly weaker statement and put an additional assumption, see Section \ref{Jump Section} below for a more general statement.

\begin{theorem}\label{1thm:main2}
Suppose that \eqref{eq:nondegeneracy2} holds but that
\begin{equation*}
  |\{u_0=0\} \cap \{g> \lambda_0 \}| >0\;.
\end{equation*}
Then there exists
\begin{equation*}
  \Lambda[u_0]:=\lim_{t\searrow 0}\lambda(t)\quad\text{ and satisfies }\quad \Lambda[u_0]>\lambda_0.
\end{equation*}
Asssume in addition  $|\{g=\Lambda[u_0]\}|=0$.
Then also $A_*^0:=\lim_{t\searrow 0}\{u(\cdot,t)>0\}$ exists (convergence in $L^1$ and Hausdorff distance sense as above) and satisfies
\begin{equation*}
  |A_*^0\setminus\{u_0>0\}|>0.
\end{equation*}
Furthermore, $\Lambda[u_0]$ and $A_*^0$ are characterized by a variational principle.
\end{theorem}

See Theorem \ref{main.p3} and Corollary \ref{cor:expl2} in Section \ref{Jump Section} for more general results and for a precise formulation of the variational principle.

\medskip
Let us discuss the assumptions \eqref{eq:nondegeneracy1-lambda} and \eqref{eq:nondegeneracy2}.
The first condition is clearly related to the condition $f\leq -\nu<0$ for the classical obstacle problem that has been present in all the regularity results stated above and that also appears as a stability condition for the free boundary
% and estimates on the symmetric difference of the support of different solutions
(see \cite{Caff81} and the exposition in \cite[Chapter 6]{Ro87}).

% Regarding the problem \eqref{eq:oblam1}-\eqref{eq:oblam4}, if $g \geq \lambda_0+\theta$, it follows from the strong
% maximum principle that $u(\cdot,t)$ becomes strictly positive for small positive times.
% Then, the interface $\partial \{u(\cdot,t)>0\}$ experiences a jump and the same is true for $\alpha$.

Note however that the right-hand side in \eqref{eq:oblam2} does not have a sign, since its integral over the support of $u$ vanishes.

\medskip
The second nondegeneracy condition \eqref{eq:nondegeneracy2} seems not to be required for problem \eqref{eq:obs-class1}-\eqref{eq:obs-class2} but appears to be quite significant for problem \eqref{eq:oblam1}-\eqref{eq:oblam4}.
In fact, in a forthcoming paper \cite{LNRVloading} we provide an example of initial data $u_0$ such that \eqref{eq:nondegeneracy1-lambda} holds while \eqref{eq:nondegeneracy2} is not satisfied.
We prove that in this case the function $\lambda$ can not be continuous at $t=0$.
The proof is rather involved and lengthy, we therefore only refer for the details to \cite{LNRVloading}.

The necessity  of this second condition in our analysis is a consequence of the particular structure of the right-hand side in \eqref{eq:oblam1}-\eqref{eq:oblam4} and its dependence on the positivity set $\{u>0\}$ through the nonlocal functional $\lambda$, whereas $f$ in \eqref{eq:obs-class1}-\eqref{eq:obs-class2} is a general function of space and time.

It is known (see for example \cite{CPS04}) that for solutions of the classical obstacle problem the boundary of the coincidence set, $\partial\{u=0\}$, has vanishing $\LL^{n+1}$-dimensional Lebesgue measure.
This implies that for almost all $t>0$ the condition \eqref{eq:nondegeneracy2} is satisfied.
However, even for the classical case it is not excluded that this condition is violated at particular times (and in particular at initial time).
Since the results in \cite{CPS04} require the assumption $f\leq -\theta<0$ for some constant $\theta>0$ (in \cite{CPS04} only $f\equiv -1$ is considered), and since the nondegeneracy condition \eqref{eq:nondegeneracy1-lambda} does not propagate to positive times the situation for \eqref{eq:oblam1}-\eqref{eq:oblam4} is even worse.

\medskip
Corresponding results to Theorems \ref{1thm:main1}, \ref{1thm:main2} can also be shown for the classical obstacle problem and seem to be new in this context.
In particular, in Theorem \ref{StefanJump} we indicate that a region where the nondegeneracy condition is initially violated becomes part of the support instantaneously.

\medskip
The results of the present paper and the arguments used in the proof may be extended to similar situations.
Key properties that are used are the continuity of the right-hand side in \eqref{eq:oblam2}, \eqref{eq:oblam3} with respect to $L^1$-convergence of the characteristic function of the support of $u$, an $L^1$-contraction property (see \eqref{contraction} below).
Beyond these general properties, however, we also exploit the particular structure, such as the appearance of a Lagrange multiplier on the right-hand side that can be expressed as a nonlocal function of the solution.
It is therefore not straightforward to formulate our result for a general class of right-hand sides.

\bigskip
The plan of this paper is the following.
First, we collect some results from previous work and reformulations of problem \eqref{eq:oblam1}-\eqref{eq:oblam4}.
In Section \ref{Continuity Section} we prove the continuity properties stated in Theorem \ref{1thm:main1}.

Discontinuity results are derived in Section \ref{Jump Section} where we in particular prove the properties stated in Theorem \ref{1thm:main2}.

Finally, in Section \ref{sec:classical} we show that the techniques developed in the current paper also yield new results for the classical parabolic obstacle problem.

\section{Previous results and preliminaries} \label{Tools}

\subsection{Assumptions and notations}
Let us state the main assumptions that we impose throughout this paper.

\begin{assumption}\label{ass:main}
Let $\Gamma\subset\R^3$ be a smooth compact surface without boundary and let $T>0$. We set  $\Gamma_T:=\Gamma \times (0,T)$.
For subsets $A\subset\Gamma$ we denote by $|A|=\calH^2(A)$ its Hausdorff measure and by $\Chi_A$ the standard characteristic function of the set $A$.

For $x_0\in\Gamma$ and $\rho>0$ we denote by $B_{\rho}(x_0)$ the ball on the surface $\Gamma$ with respect to the extrinsic (Euclidean) distance in $\R^3$.
We remark that the assumptions on $\Gamma$ imply that the intrinsic (geodesic) and the extrinsic distances induce equivalent metrics.
The Hausdorff distance between sets is denoted by $\Hdist$.

By $\Delta$ we denote the Laplace-Beltrami operator on $\Gamma$, see also the remark below.

\medskip
We assume that
\begin{equation}\label{initialdata}
  u_0\in C^2(\Gamma)\; \quad \text{with }
  \; u_0 \geq 0 \quad \text{and}
  \quad |\{u_0>0\}|>0,
\end{equation}
as well as
\begin{equation}\label{gassumptions}
  g \in C^2(\Gamma)
  \quad \text{ and } \quad 0< g_0\leq g \leq g_1 <1 \;\text{ on } \Gamma\,
\end{equation}
for some $0<g_0<g_1<1$.
We notice that in some cases, assuming only continuity for the function $g$ in the following analysis would be sufficient.
However, \eqref{gassumptions} simplifies the computations.
\end{assumption}

\begin{remark}\label{LB explained}
We recall that the relevant diffusion operator on $\Gamma$ is the corresponding Laplace-Beltrami operator, see for example	\cite{PS16}.
In local coordinates the Laplace-Beltrami operator corresponds to an elliptic operator in divergence form (with $C^2$-regular coefficients in our case).
One can deduce parabolic maximum principles in analogy to \cite[Chapter $2$]{F64} for evolution problems on $\Gamma$ involving the Laplace-Beltrami operator.
\end{remark}

\medskip
Regarding the second nondegeneracy condition, we prove in Appendix \ref{A2}, Lemma \ref{non-fat} that \eqref{eq:nondegeneracy2} is equivalent to
\begin{equation}\label{eq:nondegeneracy2a}
  \calH^2\big( \big(\{u_0>0\}\big)_{+\delta} \setminus \big(\{u_0>0\}\big)_{-\delta} \big)  \to 0   \quad \text{as } \delta \to 0\;,
\end{equation}
where
\begin{equation*}
  \big ( \{u_0>0\} \big)_{+\delta}:=\{x\;| d(x,\big ( \{u_0>0\} \big)) \leq \delta \}\;, \quad \big ( \{u_0>0\} \big)_{-\delta}:=\{x\;| d(x,\big ( \{u_0=0\} \big)) \geq \delta \}\;. \label{delta sets in intro}
\end{equation*}
In our analysis below mainly the formulation \eqref{eq:nondegeneracy2a} is used.

\subsection{Reformulations and previous results}
Before proceeding to the main analysis of this paper, we collect here some results from our previous work in \cite{LNRV21}.
In that paper and in Section \ref{Continuity Section} and Section \ref{Jump Section} below we will use the following reformulations of the problem \eqref{eq:oblam1}-\eqref{eq:oblam4}.

We define $H:\R\to\{0,1\}$, $H=\Chi_{(0,\infty)}$ as the characteristic function of the positive real numbers.

\begin{lemma}\label{2lem:equival}
Let $u_0,g$ be given as in Assumption \ref{ass:main} and consider
\begin{align*}
  u&\in L^2\big(0,T;H^1(\Gamma)\big)\cap H^1\big(0,T;H^1(\Gamma)^*\big),\\
  u&\in L^p\big(\delta,T;W^{2,p}(\Gamma)\big)\cap W^{1,p}\big(\delta,T;L^p(\Gamma)\big)\quad\text{ for any }\delta>0,\,1\leq p<\infty.
\end{align*}
Then the following are equivalent:
\begin{enumerate}
  \item\label{it:oblam} $u$ is a solution of \eqref{eq:oblam1}-\eqref{eq:oblam4} with initial data $u_0$.
  \item\label{it:oblamalt} $u$ is a solution of
  \begin{align}
    u &\geq 0 \quad &\text{ almost everywhere on }\;\Gamma_T\;, \label{eq:oblamalt2}\\
    \partial_t u -\Delta u &=-\bigg(1-\frac{g}{\lambda} \bigg )H(u) \; \quad &\text{ almost everywhere on }\;\Gamma_T\;,
    \label{eq:oblamalt1}\\
    g &\leq\lambda \quad&\text{ almost everywhere in }\{u=0\}\;, \label{eq:oblamalt3}\\
    & \lambda(t) = \fint_{\{u(\cdot,t)>0\}}g
    &\text{ for almost every }t\in (0,T)\;,
    \label{eq:oblamalt1a}\\
    u(\cdot,0) &=u_0 \quad&\text{ almost everywhere in }\Gamma\;.
    \label{eq:oblamalt4}
  \end{align}
  \item\label{it:obalph} There exists $\xi\in L^\infty(\Gamma_T)$ such that $(u,\xi)$ is a solution of
  \begin{align}
  	&\partial_t u -\Delta u  =-(1-g)\xi + \alpha g
  	&\text{ almost everywhere on } \Gamma_T\,,
  	\label{eq:obalph1}\\
    &\alpha(t)=\frac{\int_{\{u(\cdot,t)>0\}} (1-g)\,dS}{\int_{\{u(\cdot,t)>0\}} g\,dS}
    &\text{ for almost every }t\in (0,T)\;, \label{eq:obalph3}\\
  	&u\geq 0,\,  \quad  u\xi = u \, ,\quad
  	0\leq \xi\leq 1  &\text{ almost everywhere  on } \Gamma_T\,,\label{eq:obalph2}\\
    &u(\cdot,0) =u_0 &\quad\text{ almost everywhere in }\Gamma\;.
    \label{eq:obalph4}
  \end{align}
\end{enumerate}
\end{lemma}

\begin{proof}
\ref{it:oblam}$\implies$\ref{it:oblamalt}:
From \eqref{eq:oblam3} we deduce that \eqref{eq:oblamalt1} holds in $\{u>0\}$.
By the regularity assumption on $u$ and Stampacchias Lemma, \eqref{eq:oblamalt1} holds also in $\{u=0\}$, and \eqref{eq:oblam2} yields that almost everywhere in $\{u=0\}$ the inequality \eqref{eq:oblamalt3} is satisfied.

\medskip
\ref{it:oblamalt}$\implies$\ref{it:obalph}:
We define $\alpha$ by \eqref{eq:obalph3} and observe that \eqref{eq:oblamalt1a} yields
\begin{equation}
  \alpha = \frac{1}{\lambda}-1\;. \label{alphaAverage}
\end{equation}
We set
\begin{equation*}
  \xi :=
  \begin{cases}
    1 \quad&\text{ in }\{u>0\},\\
    \frac{\alpha g}{1-g} \quad&\text{ in }\{u=0\}.
  \end{cases}
\end{equation*}
Together with \eqref{eq:oblamalt1} and\eqref{alphaAverage} this implies \eqref{eq:obalph1}.

The first inequality in \eqref{eq:obalph2} holds by \eqref{eq:oblamalt2}, the second by $\xi=1$ in $\{u>0\}$ and the third since \eqref{eq:oblamalt3} yields $\alpha g \leq 1-g$ in $\{u=0\}$.

\medskip
\ref{it:obalph}$\implies$\ref{it:oblam}:
From the first inequality in \eqref{eq:obalph2} we have \eqref{eq:oblam1}.
We define $\lambda$ by \eqref{eq:oblam4}, then \eqref{eq:obalph3} yields the relation \eqref{alphaAverage}.
The inequality $\xi\leq 1$ in \eqref{eq:obalph2} implies \eqref{eq:oblam2} and the second property in \eqref{eq:obalph2} yields $\xi=1$ in $\{u>0\}$ and hence by \eqref{eq:obalph1} and \eqref{alphaAverage} that \eqref{eq:oblam3} holds.
\end{proof}

In \cite{LNRV21} we have established that for any $u_0,g$ as in Assumption \ref{ass:main} there exists a unique solution $u,\xi$ of \eqref{eq:obalph1}-\eqref{eq:obalph4} and therefore by Lemma \ref{2lem:equival} a unique solution $u$ of \eqref{eq:oblam1}-\eqref{eq:oblam4}.

By embedding theorems we have $u\in C^0([0,T];L^2(\Gamma))$ and $u\in C^{1+\beta,\frac{1+\beta}{2}}(\Gamma\times [\delta,T])$ for all $0<\beta<1$.

Furthermore, in \cite[Theorem 3.1]{LNRV21} it is shown that for any two solutions $u_1, u_2$ of \eqref{eq:obalph1}-\eqref{eq:obalph4}, the map
\begin{equation}\label{contraction}
  t \mapsto \int \limits_{\Gamma} \big ( u_1-u_2\big)_{+} (\cdot,t)\;dS \; \quad \text{is decreasing in time.}
\end{equation}
This property is crucial in the proof of Theorem \ref{1thm:main2}.

\section{Continuity results} \label{Continuity Section}

Our goal in this section is to prove Theorem \ref{1thm:main1}.

For convenience we will work in this section with the reformulation \eqref{eq:obalph1}-\eqref{eq:obalph4}, see Lemma \ref{2lem:equival}.

In \cite[Remark 2.3]{LNRV21} we have shown that the function $\xi$ is given by
\begin{equation}\label{xi1}
  \xi(\cdot,t)=
  \begin{cases}
    1 & \text{ in } \{u(\cdot,t)>0\} \\
    \frac{\alpha(t)g}{1-g} & \text{ in }\{ u(\cdot,t)=0\}
  \end{cases}
\end{equation}
for almost all $t\in (0,T)$.

By the proof of Lemma \ref{2lem:equival} the nondegeneracy assumption \eqref{eq:nondegeneracy1-lambda} is equivalent to
\begin{equation*}
 (1-g)-\alpha_0 g >0 \quad\text{ in } \{u_0=0\} \;,
\end{equation*}
where
\begin{equation*}
  \alpha_0:= \frac{\int_{\{u_0>0\}} (1-g)\,dS}{\int_{\{u_0>0\}} g\,dS}\;.
\end{equation*}
Since $\{u_0=0\}$ is compact we even obtain the uniform positivity of the left-hand side.

We will deduce Theorem \ref{1thm:main1} from the following result and the nondegeneracy assumption \eqref{eq:nondegeneracy2}.

\begin{theorem}\label{3thm:main1}
Suppose that \eqref{eq:nondegeneracy1-lambda} and \eqref{eq:nondegeneracy2} hold true and fix any $\theta>0$ such that
\begin{equation}\label{eq:nondegeneracy1}
 (1-g)-\alpha_0 g \geq\theta >0 \quad\text{ in } \{u_0=0\} \;.
\end{equation}

Then for any arbitrary small $\eta>0$, there exists $\bar t=\bar t(\eta;\theta,u_0,g)>0$ such that
\begin{equation}\label{inclusions1}
	\big(\{u_0>0 \}\big)_{-\eta} \subset \{u(\cdot,t)>0 \} \subset \big( \{u_0>0 \} \big)_{+\eta}
\end{equation}
and
\begin{equation}\label{continuityy}
  |\alpha(t)-\alpha_0| \leq \eta
\end{equation}
for all $t \in [0,\bar t]$.
\end{theorem}

We will prove this theorem in the following subsections.
As we have already mentioned in Section \ref{Tools}, problem \eqref{eq:obalph1}-\eqref{eq:obalph4} admits a unique nonnegative solution with $\alpha \in L^\infty(0,T)$.
However, beyond this regularity we have no control on $\alpha$ and its evolution could have jumps and oscillations.

\medskip
We note that it is sufficient to prove the claim for all sufficiently small $\eta>0$. In fact, if \eqref{inclusions1}, \eqref{continuityy} hold for some $\eta>0$, both properties hold  for all $\tilde \eta\geq\eta$, with $\bar t(\tilde\eta;\theta,u_0,g)=\bar t(\eta;\theta,u_0,g)$.

\medskip
A priori, the limit $\lim_{t\to 0^+}\alpha(t)$ might not exist or might be different from $\alpha_0$.
Therefore, it is convenient to consider a regularized version of \eqref{eq:obalph1}-\eqref{eq:obalph4}, for which the analog of the function $\alpha$ is smooth.

\subsection{Formulation of the regularized problem}
For $\eps>0$ we consider the following problem
\begin{align}
  \partial_t u_{\eps} - \Delta u_{\eps} &= - (1-g)f_{\eps}(u_{\eps}) +\alpha_{\eps} g \qquad &
  \text{ on } \Gamma\times (0,T)\,,\label{reg1}\\
  u_{\eps}(\cdot,0)&=u_0^{\eps}\qquad &\text{ on }\Gamma\,,\label{reg2}
\end{align}
where
\begin{equation} \label{MM term}
  f_{\eps}(u)=\frac{u}{u+\eps}
\end{equation}
describes a standard Michaelis-Menten law and the nonnegative initial data $u^\eps_0$ will be suitably constructed later (see \eqref{reg3} below).
Arguments similar to those in \cite{LNRV21} imply that a unique smooth solution of the regularized problem exists for all positive times, and that solutions approximate \eqref{eq:obalph1}-\eqref{eq:obalph4} as $\eps\to 0$.
More precisely, for any $1\leq p<\infty$ and $0<\beta<1$ we have
\begin{align}
  u_\eps &\weakto u \quad\text{ in }W^{2,1}_p(\Gamma_T), \label{conv-eps-1}\\
  u_\eps &\to u \quad\text{ in }C^{1+\beta,\frac{1+\beta}{2}}(\Gamma_T).
  \label{conv-eps-2}
\end{align}
It follows from \eqref{reg1} that $\alpha_{\eps}$ is given by
\begin{equation}\label{alphaeps}
  \alpha_{\eps}(t) = \frac{\int_{\Gamma}(1-g)f_{\eps}(u_{\eps}(\cdot,t))\,dS}{\int_{\Gamma} g\,dS}\,.
\end{equation}
We notice that $\alpha_\eps$ is Hölder continuous but that a priori there is no uniform bound on its modulus of continuity.
Moreover, the solution of \eqref{reg1}-\eqref{reg2} is strictly positive for all positive times.
We aim to prove uniform continuity estimates for the function $\alpha_{\eps}$ as $\eps \to 0^+$ for short times.

We stress that solutions to the original problem \eqref{eq:obalph1}-\eqref{eq:obalph4} and to the regularization \eqref{reg1}-\eqref{reg2} exist globally in time.
As we are only concerned with the behavior for small times, throughout this paper we will consider a fixed time interval $(0,T)$ that is  independent of $\eps$.

For further reference we also note that due to \eqref{gassumptions} we have
\begin{equation}\label{alphabound}
  0 <c(g) \leq \alpha_{\eps} ,\, \alpha \leq C(g) \qquad \text{ in }  [0,T]\,.
\end{equation}

{\bf Construction of initial data for the regularized problem:}
Due to \eqref{eq:nondegeneracy1} and the continuity of $g$, we can choose a sufficiently small $\sigma>0$ such that
\begin{equation}\label{nondegeneracy2}
 (1-g)-\alpha_0g \geq \frac{\theta}{2} \qquad \text{ in } K_\sigma:=\{u_0=0\}_{+\sigma}\;.
\end{equation}
The definition of the set $\{\; \cdot\;\}_{+ \sigma}$ is given in Definition \ref{app delta sets defn}.
Then, using \eqref{nondegeneracy2}, we can define the positive function
\begin{equation}\label{unodhat}
 {\hat u}_0^{\eps}:= \frac{\eps\alpha_0g}{(1-g) - \alpha_0g} \qquad \text{ in }
 K_{\sigma}\,.
\end{equation}
Moreover, due to \eqref{MM term}, we have that
\begin{equation}\label{unodhat2}
 -f_\eps({\hat u}_0^{\eps})(1-g) + \alpha_0g =0 \quad\text{ in }K_\sigma.
\end{equation}
We observe that
\begin{equation}
 \eps \frac{(1-\max g)g}{\max g}\leq {\hat u}_0^{\eps} \leq \eps \frac{2(1-g)-\theta}{\theta}\quad\text{ in }K_\sigma,
 \label{bounds-hatu0}
\end{equation}
where we have used $\alpha_0\geq \frac{1-\max g}{\max g}$ and \eqref{nondegeneracy2}.

We now consider a smooth cut-off function $\zeta \in C^1_c(\Gamma)$ with the following properties. We assume that, $0\leq \zeta \leq 1$ in $\Gamma$, $\zeta =1$ in $\{u_0=0\}$ and $\zeta = 0 $ in $\Gamma \backslash K_{\sigma}$ with $|\nabla \zeta | \leq \frac{\kappa}{\sigma}$ for some $\kappa>0$ which is independent of $\sigma$.
Then we take as initial data in \eqref{reg2} the function
\begin{equation}\label{reg3}
 u_0^{\eps}=  u_0+{\hat u}_0^{\eps}\zeta \qquad
 \text{ in } \Gamma\,,
\end{equation}
in particular $u_0^{\eps}\geq u_0$ on $\Gamma$ and $u_0^{\eps}={\hat u}_0^{\eps}$ in $\{u_0=0\}$.
Moreover, we observe that
\begin{equation}\label{initial reg bound}
u^\eps_0 \leq m \eps\quad \text{ in } \{u_0=0\},
\end{equation}
for some $m=m(g)>0$.

\begin{remark_nn}
We point out that $\sigma>0$ is fixed throughout the paper and does only depend on $\theta$ and the modulus of continuity of $g$.
Here its only role is to specify the size of the region in which we can define $\hat u^\varepsilon_0$ via \eqref{unodhat}.
\end{remark_nn}

We can easily check that $u_0^{\eps} \to u_0$ uniformly on $\Gamma$ as $\eps \to 0$.
Furthermore we obtain
\begin{align*}
 \alpha_{\eps}(0) \int_{\Gamma} g\,dy& = \int_{\Gamma} (1-g) f_{\eps}(u_0^{\eps})\,dy
 \to \int_{\{u_0>0\}} (1-g)\,dy + \alpha_0 \int_{\{u_0=0\}}g\,dy\\
 &= \alpha_0 \int_{\{u_0>0\}} g\,dy + \alpha_0\int_{\{u_0=0\}}g\,dy=\alpha_0\int_{\Gamma}g\,dy
\end{align*}
and hence we have that $\alpha_{\eps}(0) \to \alpha_0$ and
\begin{equation}\label{nondegeneracy3}
 (1-g(x))-\alpha_{\eps}(0)g(x) \geq\frac{ 3 \theta}{4}  \qquad\text{ for all } x \in \{u_0=0\}\,
\end{equation}
for sufficiently small $\eps>0$.

The continuity of the function $t \mapsto u_\varepsilon(\cdot,t)$ at $t=0$ in the uniform topology and \eqref{MM term}, imply that $\alpha_{\eps}$ is also continuous at $t=0$.
Combining this and the fact that $\alpha_\eps(0)\to \alpha_0$ as $\varepsilon \to 0$, we conclude that for any $\eta>0$ there exists $T_\eps>0$, in principle depending on $\eps$ such that
\begin{equation}\label{assum1}
 |\alpha_{\eps}(t)-\alpha_0|\leq \eta \quad\text{ for all }t\leq T_\eps.
\end{equation}
Recall that it is sufficient to prove the claims for all  sufficiently small $\eta>0$. In the following we will assume that $0<\eta<\frac{\theta}{4 \max g}$.
Then by \eqref{eq:nondegeneracy1} and \eqref{assum1} we conclude that
\begin{equation}\label{nondegeneracy4}
 (1-g(x))-\alpha_{\eps}(t)g(x) \geq\frac{\theta}{2} \qquad\text{ for all } x \in \{u_0=0\} \text{ and }
 t \in [0,T_\eps]\,.
\end{equation}
Our goal will be to show that there exists a time $\bar t>0$ that is independent of $\eps$ such that \eqref{assum1} and hence
\eqref{nondegeneracy4} hold also in $[0,\bar t]$.
We will use this result to show that
for some $L_0>0$ and for $t \in [0,\bar t]$ the sets $\{u_{\eps}(\cdot,t) > L_0\eps\}$ and
$\{u_0^{\eps}>L_0\eps\}$ are close in the sense of Lebesgue measure for small $t$.
Then we can pass to the limit $\eps\to0$ to conclude the same for
$\{u(\cdot,t)>0\}$ and $\{u_0>0\}$.
Recalling the form of $\alpha$ in \eqref{eq:obalph3} we see that this is the key estimate in order to prove the continuity of $\alpha$.

\subsection{Uniform continuity of $\alpha_\eps$ at $t=0$}
\label{S.tzero}

We start with a few auxiliary lemmas.
The dependence of constants on the data $u_0$ and $g$, also through the  uniform bounds on $\alpha_{\eps}, \alpha$ in \eqref{alphabound}, will not be written explicitly in the following.
However, any additional dependence will be specified each time.

Furthermore  $u_{\eps}$ always denotes a solution to \eqref{reg1}-\eqref{reg2} with data $u^\varepsilon_0$ as in
\eqref{reg3} and $T_\eps>0$ such that \eqref{nondegeneracy4} holds.

\bigskip

Our first result yields uniform continuity in $\eps$ for short times.

\begin{lemma}\label{L.uniformestimate}
Let $u_{\eps}$ be the unique solution to \eqref{reg1}-\eqref{reg2} with data $u^\varepsilon_0$ as in \eqref{reg3}.
Then,
\begin{equation}\label{uniformestimate}
  \|u_{\eps}(\cdot,t)-u_0^{\eps}\|_{L^{\infty}(\Gamma)} \leq C_1 t^{\frac{1+\beta}{2}} \qquad \text{ for all } t\leq T
\end{equation}
and for all $ 0<\beta <1\;$.
\end{lemma}

\begin{proof}
Since $u_{\eps}$ is a solution to \eqref{reg1}-\eqref{reg2} with data as in \eqref{reg3} and the right-hand side of \eqref{reg1} is uniformly bounded in $\eps$, we deduce from standard
H\"older regularity results for parabolic equations, see \cite{F64}, that $u_{\eps} \in C^{1+\beta,\frac{1+\beta}{2}}(\Gamma_{T})\;$ for any $0<\beta<1\;$.
The claim then follows.
\end{proof}

 To proceed, we define for the sake of simplicity
\begin{equation}
  U:=\{ u_0 =0 \}\;, \quad  V:= \Gamma \backslash U = \{ x \,|\, u_0(x)>0\}\, \label{U,V defn}
\end{equation}
and for any sufficiently small $\delta>0$  we define
\begin{equation}\label{Ud,Vd}
  U_{-\delta} :=\big ( \{u_0=0\} \big)_{-\delta}=\{x\;| d(x,\big ( \{u_0>0\} \big)) \geq \delta \}, \quad V^\delta:= \{u_0\geq \delta\} \;.
\end{equation}

A key property in the analysis of free boundary problems is the so-called nondegeneracy property.

This property states that if a solution to \eqref{eq:obs-class1}-\eqref{eq:obs-class3} is small in a sufficiently large open set, then it vanishes in a smaller set (cf.
\cite{BF76}).
A version of this property for the stationary solutions to \eqref{eq:obalph1}-\eqref{eq:obalph3}, has been formulated in \cite[Proposition 3.9(5)]{NRV20}.
The next lemma yields a variation of this nondegeneracy property for the regularized problem \eqref{reg1}, \eqref{reg2}.

Since the solution to \eqref{reg1}, \eqref{reg2} is strictly positive due to the maximum principle, the corresponding nondegeneracy result is formulated as follows.
If $u_\eps$ is smaller than some number independent of $\eps$, in a sufficiently large set with size independent of $\eps$, then $u_\eps \leq L_0\eps$ for some positive constant $L_0$ which is independent of $\eps$, in a smaller set with size independent of $\eps$ as well.
It is worth noticing that in the limit $\eps \to 0^+$, this result would "converge" to the standard nondegeneracy result for the Stefan problem, cf. \cite[Theorem $3.1$]{BF76}.

\begin{lemma}\label{L.usmall}
Consider $T_\eps>0$ such that \eqref{assum1} holds.
Then there exist positive constants $\rho_{\max}=\rho_{\max}(\Gamma)$, $A=A(\Gamma)$ and $L_0=L_0(g,\theta)$ such that for any $ \tilde t \in [0,T_\eps]$ and any $\rho\in (0,\rho_{max})$ the following holds:

If $B_{2\rho}(x_0)\subset U$ and if
\begin{equation*}
  u_{\eps} \leq \frac{\theta}{A} \rho^2 \qquad \text{ in } B_{2 \rho}(x_0)\times [0,\tilde t]\,,
\end{equation*}
then $u_\eps$ satisfies
\begin{equation*}
  u_{\eps} \leq L_0\eps \qquad \text{ in } B_{\rho}(x_0)\times [0,\tilde t ].
\end{equation*}
\end{lemma}

\begin{proof}
Without loss of generality  $x_0=0\in\Gamma$.
Let us also recall \eqref{initial reg bound}.
Then we can choose $L_0>0$ such that
\begin{equation}\label{L0choice}
  L_0\geq 2 m \quad \text{and} \quad (1-g)\frac{1}{L_0+1}\leq \frac{\theta}{4}\;.
\end{equation}
We proceed by contradiction, that is we suppose that there exists $(y,\tau ) \in B_{\rho}(0)\times [0,\tilde t]$ such that $u_{\eps}(y,\tau) > L_0\eps$.
As a candidate for a supersolution we define
\begin{equation*}
  \tilde u(x,t):=L_0\eps +  \frac{\theta}{A}|x-y|^2  + \frac{\theta}{8}(\tau-t)
\end{equation*}
and compute  that
\begin{equation*}
\partial_t \tilde u - \Delta \tilde u = -\frac{\theta}{8} - \frac{4 \theta}{A} - \frac{2\theta}{A} \vec H \cdot (x-y) \geq - \frac{\theta}{4}
\end{equation*}
if $\rho \leq \rho_{\max}(\Gamma)$  is sufficiently large.
(Here $\vec H$ denotes the mean curvature vector on $\Gamma$.)  By \eqref{nondegeneracy4} and the choice
of $L_0$ we have in $(B_{2 \rho} (0)\times [0,\tau]) \cap \{u_{\eps} \geq L_0\eps\}$ that
\begin{equation*}
 \partial_t u_{\eps} - \Delta u_{\eps} = -(1-g)f_{\eps}(u_{\eps}) + \alpha_{\eps}(t)g \leq - \frac{\theta}{4} \leq \partial_t \tilde u - \Delta \tilde u\;.
\end{equation*}

Furthermore, we check that on $(\partial B_{2 \rho}(0) \times [0,\tau]) \cap \{u_{\eps} \geq L_0\eps\}$ we have
\begin{equation*}
 u_{\eps} \leq \frac{\theta}{A} \rho^2 \leq \tilde u
\end{equation*}
while
on $(B_{2 \rho}(0) \times [0,\tau]) \cap \partial \{u_{\eps} \geq L_0\eps\}$ it clearly holds that $u_{\eps} - \tilde u \leq 0$.
Furthermore $u_{\eps}(\cdot,0) \leq \tilde u(\cdot,0)$ by our choice of $L_0$ in \eqref{L0choice}.
Hence the parabolic maximum principle implies $u \leq \tilde u$ in $(B_{2\rho}(0)\times [0,\tau])\cap
\{u_{\eps} \geq L_0\eps\}$
which gives a contradiction.
\end{proof}

Now, we can prove the left inclusion in \eqref{inclusions1} for the solution of the regularized problem \eqref{reg1}, \eqref{reg2}.

\begin{corollary}\label{L.inclusion}
Let $\delta>0$, $\beta \in (0,1)$ be fixed and $T_\eps>0$ be such that \eqref{assum1} holds.
Let $A=A(\Gamma)$ be as in Lemma \ref{L.usmall}, $C_1>0$ as in Lemma \ref{L.uniformestimate} and set
\begin{equation}\label{tstardef}
  t^*(\delta):= \bigg ( \frac{\theta \delta^2}{8 A C_1} \bigg )^{\frac{2}{1+\beta}}.
\end{equation}
Then, there exists $\eps_0=\eps_0(\delta,u_0,g) > 0$, such that for all $0<\eps \leq \eps_0$ we have
\begin{equation*}
  U_{-\delta} \subset \{u_{\eps}(\cdot,t) \leq L_0\eps\}
  \quad\text{ for all }\quad 0 \leq t \leq \min\{T_\eps, t^*(\delta)\},
\end{equation*}
where $U_{-\delta}$ is given by \eqref{Ud,Vd}.
\end{corollary}

\begin{proof}
Due to \eqref{initial reg bound}  Lemma \ref{L.uniformestimate} implies
\begin{equation*}
 \|u_{\eps}(\cdot,t)\|_{L^{\infty}(U)} \leq \| u_{\eps}(\cdot,t)-u_0^{\eps}\|_{L^{\infty}(U)} +\|u_0^{\eps} \|_{L^{\infty}(U)}\leq C_1t^{\frac{1+\beta}{2}} + m\eps\,.
\end{equation*}
Hence, we deduce by \eqref{L0choice} that
\begin{equation*}
  \|u_{\eps}(\cdot,t)\|_{L^{\infty}(U)} \leq 2m\eps \leq L_0\varepsilon
  \quad\text{ for all } 0\leq t\leq \min\Big\{ T_\eps, t^*(\delta), \Big(\frac{m\eps}{C_1} \Big )^{\frac{2}{1+\beta}}\Big\}\;.
\end{equation*}
On the other hand, for $(\frac{m\eps}{C_1} \big )^{\frac{2}{1+\beta}}<t\leq \min\{T_\eps, t^*(\delta)\}$ we obtain by \eqref{tstardef} that
\begin{equation*}
  \|u_{\eps}(\cdot,t)\|_{L^{\infty}(U)} \leq 2C_1\big(\min\{T_\eps, t^*(\delta)\}\big)^{\frac{1+\beta}{2}} \leq \frac{\theta}{A} \Big(\frac{\delta}{2}\Big)^2\;.
\end{equation*}
We now apply Lemma \ref{L.usmall}.
To this end we choose $\rho = \min(\rho_{\max},\delta/2)$ and arbitrary $x_0 \in U_{-\delta}$.
Then Lemma \ref{L.usmall} implies the claim.
\end{proof}

As long as \eqref{assum1} is valid, we obtain in the next lemma some detailed pointwise estimates for the function $f_\varepsilon(u_\varepsilon)$ which is given by \eqref{MM term}.

\begin{lemma}\label{L.alphepsapprox}
Let $\eta \in \big(0, \frac{\theta}{4\max g}\big]$ and $T_\eps>0$ be such that \eqref{assum1} holds.
Then for given small $\delta>0$ there exists a constant $C_{\delta}>0$ such that in the set $U_{-2\delta} \times (0,T_\eps)$ we have
\begin{equation}\label{estimate1}
  f_\varepsilon(u_\varepsilon) (1-g) - (\alpha_0+ \eta) g \leq  C_{\delta} \eps
\end{equation}
and
\begin{equation}\label{estimate2}
  f_\varepsilon(u_\varepsilon) (1-g) - (\alpha_0- \eta) g \geq - C_{\delta} \eps\,.
\end{equation}
\end{lemma}

\begin{proof}
	We are going to  prove \eqref{estimate1}, the proof of \eqref{estimate2} goes analogously.

	{\bf Step 1:} We first construct a suitable supersolution by defining
	${\bar u}_{\eps}\colon U_{-\delta} \times [0,T_\eps] \to \R_+$  via
	\begin{align}
  	\partial_t {\bar u}_{\eps} - \Delta {\bar u}_{\eps}&= -(1-g) f_{\eps}\big( {\bar u}_{\eps})
  	+ (\alpha_0+\eta) g\,, \qquad &\text{ in } U_{-\delta} \times (0,T_\eps) \label{baru1}\\
  	{\bar u}_{\eps}(\cdot,0)&= \eps \frac{(\alpha_0+\eta)g}{(1-g)-(\alpha_0+\eta)g} \qquad &\text{ in }U_{-\delta}
  	\,, \label{baru2}\\
  	{\bar u}_{\eps}&=L_0 \eps \qquad \qquad\qquad \qquad &\text{ on } \partial U_{-\delta} \times (0,T_\eps)\,.\label{baru3}
	\end{align}
	Indeed, for $w_{\eps}:= u_{\eps} - {\bar u}_{\eps}$ we find, due to \eqref{assum1}, that
	\begin{align*}
  	\partial_t w_{\eps}-\Delta w_{\eps} & \leq - (1-g) \big(f_{\eps}(u_{\eps}) - f_{\eps} ({\bar u}_{\eps})\big)\\
  	& = - (1-g) \frac{\eps}{(\eps+u_\eps)(\eps+\bar u_\eps)} w_{\eps}\;.
	\end{align*}
	Furthermore, on $U_{-\delta}$ it holds
	\begin{equation*}
  	u_{\eps}(\cdot, 0) \leq \eps \frac{(\alpha_0+\eta)g}{(1-g) - (\alpha_0+\eta)g} = {\bar u}_{\eps}(\cdot, 0)\,,
	\end{equation*}
	while on $\partial U_{-\delta} \times [0,T_\eps]$ we have $w_{\eps} \leq 0$ by Corollary \ref{L.inclusion}.
	This yields $w_{\eps} \leq 0$, hence
	\begin{equation}
  	u_\eps \leq {\bar u}_\eps\quad\text{ in }U_{-\delta}\times [0,T_\eps].
  	\label{3eq:1007}
	\end{equation}

	{\bf Step 2:} We next show that for some $C=C(\delta,L_0)$ it holds
	\begin{equation}\label{baruest}
  	|{\bar u}_{\eps}(x,t) - {\bar u}_{\eps}(x,0)| \leq C\eps^2 \qquad \text{ in }
  	U_{-2\delta} \times [0,T_{\eps}]\,.
  	\end{equation}
	Further we consider $U_\eps^0:=\frac{1}{\eps}{\bar u}_0^{\eps}$, $U_\eps:=\frac{1}{\eps}{\bar u}_{\eps}$ and $v_\eps:=U_\eps - U_\eps^0$.
	We compute, with $f(u)=\frac{u}{u+1}$, that
	\begin{align*}
  	\partial_t v_{\eps} - \Delta v_{\eps} &= - \frac{1}{\eps} (1{-}g)f\big( U_{\eps}\big) +
  	\frac{1}{\eps} (\alpha_0+\eta)g
  	- \Delta U_\eps^0(x)\\
  	&= - \frac{1}{\eps} (1-g) \big( f(U_{\eps})-f(U_\eps^0)\big) - \Delta U_\eps^0(x)\\
  	& = - \frac{1}{\eps}(1-g) \frac{1}{(1+U_\eps)(1+U_\eps^0)} v_{\eps} -\Delta U_\eps^0(x)\,.
	\end{align*}
	We first notice that, since $|\Delta U_\eps^0|\leq C_0$ for some $C_0=C_0(g)$, we have that $|v_\eps|$ and then also $U_{\eps}$ are uniformly bounded in $\eps$ by some constant only depending on $\delta, L_0$.
	Hence, we have
	\begin{equation}
  	\partial_t v_{\eps} - \Delta v_{\eps} \leq -\frac{\mu}{\eps}v_\eps + C_0
	\end{equation}
	for some $\mu=\mu(\delta,L_0)>0$.

	We compare $v_\eps$ with the solution of the boundary value problem
	\begin{equation}
  	- \eps\Delta w_{\eps} = -\mu w_\eps + \eps C_0\quad\text{ in }U_{-\delta},\qquad
  	w_{\eps} = v_\eps \quad\text{  on }\partial U_{-\delta}.
  	\label{3eq:1003}
	\end{equation}
	By applying maximum principles we find that in $U_{-\delta}$ for some $C_1=C_1(g,L_0)$
	\begin{equation}
  	0\leq w_\eps\leq C_1,\qquad v_\eps(\cdot,t) \leq w_\eps\quad\text{ for all }0<t<T_\eps.
  	\label{3eq:1006}
	\end{equation}
	We next claim that there exists $\Lambda=\Lambda(\mu,\delta)$ such that
	\begin{equation}
  	w_\eps \leq \Lambda\eps \quad\text{ in }U_{-2\delta}.
  	\label{3eq:1002}
	\end{equation}
	To prove this estimate consider an arbitrary $\Lambda>0$ (to be chosen sufficiently large later) and assume by contradiction that $w_\eps(x_0)>\Lambda\eps$ for some $x_0\in U_{-2\delta}$.

	Let $\vartheta:= \eps \frac{\Lambda}{2C_1}+\varphi^2$ where $\varphi\in C^\infty_c(B_{\delta}(x_0))$ is chosen such that
	\begin{equation*}
  	\vartheta(x_0)=1,\quad \vartheta\leq 1,\quad \|\vartheta\|_{C^2(B_{\delta}(x_0))}\leq C(\delta).
	\end{equation*}
	Observe that $\vartheta\geq  \eps \frac{\Lambda}{2C_1}$ and
	\begin{equation}
  	\frac{|\nabla\vartheta|^2}{\vartheta} \leq \frac{4\varphi^2|\nabla\varphi|^2}{\eps \frac{\Lambda}{2C_1}+\varphi^2} \leq C(\delta)
  	\label{3eq:1004}
	\end{equation}
	for some $C(\delta)$ independent of $\eps,\Lambda$.
	Furthermore, we compute
	\begin{equation*}
  	\nabla(\vartheta w_\eps)=\vartheta\nabla w_\eps + w_\eps\nabla \vartheta,\qquad
  	\Delta (\vartheta w_\eps) = \vartheta \Delta w_\eps + 2\nabla\vartheta\cdot\nabla w_\eps +w_\eps\Delta\vartheta,
	\end{equation*}
	and deduce from \eqref{3eq:1003}
	\begin{align*}
  	C_0\vartheta &= -\Delta(\vartheta w_\eps) +2\nabla\vartheta\cdot\nabla w_\eps+w_\eps\Delta\vartheta +\frac{\mu}{\eps} (\vartheta w_\eps)\\
  	&= -\Delta(\vartheta w_\eps) +\frac{2}{\vartheta}\nabla\vartheta\cdot\nabla(\vartheta w_\eps) +\Big(\frac{\mu}{\eps}-\frac{2|\nabla\vartheta|^2}{\vartheta^2}+\frac{1}{\vartheta}\Delta\vartheta\Big)(\vartheta w_\eps).
	\end{align*}
	Using \eqref{3eq:1004} yields for all $\Lambda\geq \Lambda_0$, $\Lambda_0=\Lambda_0(\delta,C_1,\mu)$
	\begin{align*}
  	-\frac{2|\nabla\vartheta|^2}{\vartheta^2}+\frac{1}{\vartheta}\Delta\vartheta
  	& \geq -\frac{C(\delta)}{\vartheta}\geq -\frac{2C_1C(\delta)}{\eps\Lambda} \geq -\frac{\mu}{2\eps},
	\end{align*}
	hence
	\begin{align}
  	C_0\vartheta &\geq -\Delta(\vartheta w_\eps)  +\frac{2}{\vartheta}\nabla\vartheta\cdot\nabla(\vartheta w_\eps) +\frac{\mu}{2\eps}(\vartheta w_\eps).
  	\label{3eq:1005}
	\end{align}
	By the choice of $\vartheta$ we have $(\vartheta w_\eps)(x_0)>\Lambda\eps>(\vartheta w_\eps)|_{\partial B(x_0,\delta)}$.
  Hence $\vartheta w_\eps$ attains an interior maximum at some $x_1\in B(x_0,\delta)$.
	Evaluating \eqref{3eq:1005} we deduce that
	\begin{align*}
  	C_0  \geq C_0 \vartheta(x_1) &\geq \frac{\mu}{2\eps}(\vartheta w_\eps)(x_1) \geq \frac{\mu}{2\eps}(\vartheta w_\eps)(x_0) > \frac{\mu\Lambda}{2},
	\end{align*}
	a contradiction for $\Lambda\geq \Lambda_*$, where  $\Lambda_*=\Lambda_*(C_0,C_1,\mu,\delta)$ only depends on $\delta$ and $L_0$.

	This completes the proof of \eqref{3eq:1002}.

	{\bf Step 3:} Using \eqref{3eq:1007}, \eqref{baru2}, \eqref{baruest} and the monotonicity and continuity of $f$ we finally obtain
	\begin{align*}
  	f_\varepsilon(u_\varepsilon) (1-g) - (\alpha_0+ \eta) g
  	&\leq f_\varepsilon(\bar u_\varepsilon) (1-g) - (\alpha_0+ \eta) g\\
  	&= f_\varepsilon(\bar u_\varepsilon^0) (1-g) - (\alpha_0+ \eta) g +
  	\frac{\eps(\bar u_\eps -\bar u_\eps^0)}{(\eps+\bar u_\eps)(\eps+\bar u_\eps^0)}(1-g)\\
  	&\leq C(\delta,L_0)\eps.
	\end{align*}
	Since the choice of $L_0$ only depends on $g,\theta$ this proves \eqref{estimate1}.
\end{proof}

The following proposition plays a crucial role in proving Theorem \ref{3thm:main1}.
It states that we can choose $T_\eps>0$ in \eqref{assum1} independent of $\eps$.
The key idea in this proof is to consider the maximal time interval where \eqref{assum1} is satisfied and show that it could be extended unless it contains and interval $[0,\bar t]$, with $\bar t$ independent of $\eps$.

\begin{proposition}\label{P.main1}
Let $\theta>0$ as in \eqref{eq:nondegeneracy1} and  $0<\eta<\frac{\theta}{4 \max g}$.
Then, for sufficiently small $\eps>0$, there exists $\bar t:=\bar t(\eta,\theta,u_0,g)>0$ that is independent of $\eps$ such that
\begin{equation*}
  |\alpha_{\eps}(t)-\alpha_0| \leq \eta \quad \text{for all } t \in [0,\bar t]\;.
\end{equation*}
In particular, we can choose $T_\eps=\bar t$ in Corollary \ref{L.inclusion} independent of $\eps>0$.
\end{proposition}

\begin{proof}
 {\bf Step 1:}
 We observe that Lemma \ref{L.alphepsapprox} yields that there exists $r_1 \in [-\eta, \eta]$ such that
\begin{equation}\label{hypothesi1}
  \int_{U} (1-g) f_{\eps} \big (u_{\eps}(\cdot,t) \big)\,dS = \big(\alpha_0+r_1 \big) \int_U g\,dS\;+ O(C_\delta \eps) + \omega(\delta),
\end{equation}where $\omega(\delta) \to 0$ as $\delta \to 0$.
Moreover, by \eqref{reg3} we have that $u^\eps_0 \geq u_0$ on $\Gamma$.
Using this and Lemma \ref{L.uniformestimate} and possibly passing to a smaller value of $t^*(\delta)$, we obtain for all $t \leq t^*(\delta)$ that
\begin{equation}\label{hypothesis2}
  \int_{V^\delta} (1-g) f_{\eps} \big (u_{\eps}(\cdot,t)\big)\,dS=\int_{V^\delta}(1-g)\,dS + O\Big( \frac{\eps}{\delta}\Big)\;.
\end{equation}
Moreover, using \eqref{eq:nondegeneracy2a}, \eqref{gassumptions} and \eqref{MM term} we deduce
\begin{equation}\label{aporroia}
 0\leq \int_{\Gamma \backslash (U\cup V^\delta)} (1-g) f_{\eps} \big (u_{\eps}(\cdot,t)\big)\,dS  \leq \vert \Gamma \backslash (U\cup V^\delta) \vert\,\to\, 0\,, \;\text{ as } \delta \to 0.
\end{equation}
Therefore we can write, combining \eqref{hypothesi1}, \eqref{hypothesis2},
\begin{align*}
  \alpha_{\eps}(t) \int_{\Gamma} g\,dS& =  \int_{V^\delta} (1-g) f_{\eps} \big (u_{\eps}(\cdot,t)\big)\,dS +
  \int_{U} (1-g) f_{\eps} \big (u_{\eps}(\cdot,t) \big)\,dS \\
  &\qquad\qquad + \int_{\Gamma \backslash (U\cup V^\delta)} (1-g) f_{\eps} \big (u_{\eps}(\cdot,t)\big)\,dS\\
  &=\int_{V^\delta}(1-g)\,dS + O\Big( \frac{\eps}{\delta}\Big) + \big(\alpha_0+r_1 \big) \int_U g\,dS + r_2\;,
\end{align*}
with $\vert r_2 \vert  \leq C_\delta \eps +\omega(\delta) $.
This estimate can be further simplified as follows
\begin{align*}
  \alpha_{\eps}(t) \int_{\Gamma} g\,dS& =
  \int_{V}(1-g)\,dS + \big(\alpha_0+r_1\big) \int_U g\,dS + r_3\\
   &= \alpha_0\int_{\Gamma\backslash U} g\,dS + \alpha_0\int_{U}g\,dS
   +r_1 \int_U g\,dS + r_3\\
   &=\alpha_0\int_{\Gamma} g\,dS +r_1 \int_U g\,dS + r_3\,,
\end{align*}
where $\vert r_3 \vert \leq r_2+O\Big( \frac{\eps}{\delta}\Big)+\vert V\setminus V^\delta \vert$.
Thus, we obtain for all $0 \leq t \leq T_\eps$
\begin{equation*}
 \big| \alpha_{\eps}(t) - \alpha_0\big| \leq \kappa \vert r_1 \vert  + C\vert r_3 \vert \leq (1-\kappa) \eta+C\omega(\delta)+C_\delta \eps
\end{equation*}
with $\kappa := \frac{\int_{\{u_0>0\}}g\,dS}{\int_{\Gamma} g\,dS}>0$.
This is the key estimate in order to complete the proof.

\medskip
Indeed, let us first fix $\delta$ depending on $\eta$ such that $C\omega(\delta)\leq \frac{\kappa}{4}\eta$ and then $\eps$ depending on $\delta$ sufficiently small such that $C_\delta \eps\leq \frac{\kappa}{4}\eta$.
Then it holds
\begin{equation}
 \big| \alpha_{\eps}(t) - \alpha_0\big| \leq \Big(1-\frac{\kappa}{2}\Big)\eta
 \quad\text{ for all }0 \leq t \leq T_\varepsilon.
\label{eq:kappa}
\end{equation}
Next, define
\begin{equation*}
 \tilde T_\eps:=\sup \big\{0\leq s\leq t^*(\delta)\,:\, |\alpha_{\eps}(t)-\alpha_0| \leq \eta \; \text{ for all } t \in [0,s]\big\}\,.
\end{equation*}
If $\tilde T_\eps<t^*(\delta)$ the continuity of $\alpha_\eps$ and \eqref{eq:kappa} imply that $|\alpha_{\eps}(t)-\alpha_0| \leq \eta$ holds on an interval $[0,T^\dagger]$ with $\tilde T_\eps<T^\dagger\leq t^*(\delta)$, a contradiction to the definition of $\tilde T_\eps$.

Thus, setting $\bar t:=\bar t(\eta)=t^*(\delta)$, we obtain that $\vert \alpha_{\eps}(t)-\alpha_0 \vert \leq \eta$ in $[0,\bar t]$.
\end{proof}

Proposition \ref{P.main1} implies that there exists a time $\bar t>0$ which is independent of
$\eps$ such that \eqref{assum1} and hence \eqref{nondegeneracy4} holds for all $t \in [0,\bar t]$.

\subsection{Proofs of Theorem \ref{3thm:main1} and Theorem \ref{1thm:main1}}

With the uniform continuity result stated in Proposition \ref{P.main1} we can now prove Theorem \ref{3thm:main1} by passing to the limit $\eps \to 0$.

\begin{proof}[Proof of Theorem \ref{3thm:main1}]
Recalling \eqref{eq:obalph3} and \eqref{alphabound}, we estimate
 \begin{align}
  |\alpha(t)-\alpha_0|&= \bigg| \frac{\int_{\{u(\cdot,t)>0\}} (1-g)\,dS}{\int_{\{u(\cdot,t)>0\}}g\,dS}
  - \frac{\int_{\{u_0>0\}} (1-g)\,dS}{\int_{\{u_0>0\}}g\,dS}\bigg| \notag \\
 & \leq \frac{1}{\int_{\{u_0>0\}} g\,dS} \Big| \int_{\{u_0>0\}} (1-g)\,dS - \int_{\{u(\cdot,t)>0\}}
  (1-g)\,dS \Big| \notag \\
  &\qquad
  + \frac{\alpha(t)}{\int_{\{u_0>0\}} g\,dS}
  \Big| \int_{\{u(\cdot,t)>0\}} g\,dS -
  \int_{\{u_0>0\}}g\,dS \Big| \notag \\
  & \leq C \big |\{u_0>0\}\Delta \{ u(\cdot,t)>0\} \big |
  \label{basic1}
 \end{align}
Then it is sufficient to show that for all sufficiently small $\delta>0$ there exists $t^\dagger(\delta)>0$ such that
\begin{equation}\label{inclusions2}
 V^\delta \subset \{u(\cdot,t)>0\} \subset  \big(U_{-\delta}\big)^{\mathsf{c}}
 \quad\text{ for all }0<t<t^\dagger(\delta)\;,
\end{equation}
where $V^\delta,U_{-\delta}$ are given by \eqref{Ud,Vd}.

\medskip
Indeed, by \eqref{U,V defn} we observe that
\begin{equation*}
  V^\delta \subset V \subset \big(U_{-\delta}\big)^{\mathsf{c}}\;.
\end{equation*}
This in turn yields that for all $0<t<t^\dagger(\delta)$
\begin{equation}
  \{u(\cdot,t)>0 \} \Delta \{u_0>0 \}  \subset \big(U_{-\delta}\big)^{\mathsf{c}} \setminus  V^\delta\;.
  \label{eq:basic2a}
\end{equation}
Using Lemma \ref{-delta set}, we infer that for any $\delta>0$ we can choose $\bar \delta:=\bar \delta(\delta)\geq\delta$ with $\bar \delta \to 0$ as $\delta \to 0$ such that $V_{-\bar\delta}\subset V^\delta$.
By Lemma \ref{delta-sets-complement} it holds $\big(U_{-\delta}\big)^{\mathsf{c}}\subset V_{+\bar \delta}$, which yields
\begin{equation}\label{basic2}
  \big(U_{-\delta}\big)^{\mathsf{c}} \setminus  V^\delta \subset V_{+\bar \delta} \setminus  V_{-\bar \delta}\;.
\end{equation}
Moreover, we deduce by \eqref{inclusions2} that
\begin{equation}\label{inclusions3}
 V_{-\bar \delta} \subset \{u(\cdot,t)>0\} \subset  V_{+\bar \delta}
 \quad\text{ for all }0<t<t^\dagger(\delta)\;.
\end{equation}
Finally, \eqref{eq:basic2a}, \eqref{basic2}, the convergence $\bar \delta \to 0$ as $\delta \to 0$ and \eqref{eq:nondegeneracy2a} yield that
\begin{align}
  &| \{u(\cdot,t)>0 \} \Delta \{u_0>0 \} | \leq | \{u_0>0\}_{+\bar \delta} \setminus  \{u_0>0\}_{-\bar \delta} |
  \quad\text{ for all }0<t<t^\dagger(\delta)\;,
  \label{eq:L1-conv-supp1}\\
  &| \{u_0>0\}_{+\bar \delta} \setminus  \{u_0>0\}_{-\bar \delta} | \to 0 \quad \text{as } \delta \to 0\;.
  \label{eq:L1-conv-supp2}
\end{align}
% \begin{equation*}
%   \sup_{0<t<t^\dagger(\delta)}| \{u(\cdot,t)>0 \} \Delta \{u_0>0 \} | \leq | \{u_0>0\}_{+\bar \delta(\delta)} \setminus  \{u_0>0\}_{-\bar \delta(\delta)} | \to 0 \quad \text{as } \delta \to 0\;.
% \end{equation*}
Therefore, for given $\eta>0$, we can choose a sufficiently small $0<\delta\leq\eta$, $\delta=\delta(\eta)$, such that the right-hand side of \eqref{basic1} is less or equal than $\eta$ for all $0<t<\bar{t}(\eta)$, with $\bar{t}(\eta)=t^\dagger(\delta)$.
Thus, we obtain \eqref{continuityy}.
Moreover, by \eqref{inclusions3} and $\delta\leq\eta$ we obtain \eqref{inclusions1}.

\medskip
This proves that if \eqref{inclusions2} holds, then Theorem \ref{3thm:main1} follows.
To this end, we now proceed to the proof of \eqref{inclusions2}.

\medskip
On the one hand, by Corollary \ref{L.inclusion} and Proposition \ref{P.main1} we obtain that for any $\delta >0$ and for any $\eps=\eps(\delta)$ sufficiently small, $U_{-\delta} \subset \{u_\eps \leq L_0\eps \}$ for all $t\in [0,t^*(\delta)]$.
This in particular yields that for any $x \in U_{-\delta}$ and for all $0\leq t \leq t^*(\delta)$, $u_\eps(x,t) \leq L_0\eps$.
Due to the uniform convergence $u_\eps \to u$, we conclude that $U_{-\delta} \subset \{u(\cdot,t)=0 \}$ for all $t$ in $[0,t^*(\delta)]$.
Taking the complements, the right inclusion follows.

On the other hand, $x\in V^\delta$ implies $u_0(x)\geq \delta$ and by uniform convergence $u(\cdot,t)\to u_0$, see Section \ref{Tools}, and after possibly passing to a lower value of $t^*(\delta)$, we have that $u(x,t)>0$ for all $t\leq t^*(\delta)$.
Thus $V^\delta \subset \{u(\cdot,t)>0\}$ for all $t\leq t^*(\delta)$.

Hence, choosing $t^\dagger(\delta)=t^*(\delta)$ we deduce \eqref{inclusions2}.
\end{proof}

\begin{proof}[Proof of Theorem \ref{1thm:main1}]
The right continuity of $\lambda$ at $t=0$ follows directly from \eqref{continuityy} and the relation \eqref{alphaAverage}.

The right-continuity of the support in the sense of $L^1$-convergence of the associated characteristic functions is a consequence of \eqref{eq:L1-conv-supp1}, \eqref{eq:L1-conv-supp2}.
To also prove the Hausdorff-convergence of supports we argue by contradiction.
Assume
\begin{equation}
  \Hdist(\{u(\cdot,t)>0\},\{u_0>0\})\centernot\longrightarrow 0\quad\text{ as }t\searrow 0.
  \label{eq:notHausdorff}
\end{equation}
Then there exists $\eps>0$ and a sequence $(t_k)_k$ of real numbers with $t_k\searrow 0$ as $k\to\infty$ such that
\begin{equation*}
  \Hdist(\{u(\cdot,t_k)>0\},\{u_0>0\})>\eps \quad\text{ for all }k\in\N.
\end{equation*}
By \eqref{inclusions1} we have $\{u(\cdot,t_k)>0\}\subset (\{u_0>0\})_{+\eps}$ for all sufficiently large $k\in\N$.
By \eqref{eq:notHausdorff}, without loss of generality for the whole sequence,
\begin{equation}
  \{u_0>0\}) \not\subset (\{u(\cdot,t_k)>0\})_{+\eps}
  \quad\text{ for all }k\in\N\;.
  \label{eq:notHausd}
\end{equation}
This yields the existence of a sequence $(x_k)_k$ in $\{u_0>0\}$ with
\begin{equation*}
  x_k\not\in (\{u(\cdot,t_k)>0\})_{+\eps}
  \quad\text{ for all }k\in\N\;.
\end{equation*}
For a subsequence (not relabled) and a point $x_0\in \overline{\{u_0>0\}}$ we have $x_k\to x_0$ and deduce
\begin{equation}
  B(x_0,\frac{\eps}{2})\cap \{u(\cdot,t_k)>0\} = \emptyset
  \quad\text{ for all }k\gg 1\;.
  \label{eq:notHausd2}
\end{equation}
Since $x_0\in \overline{\{u_0>0\}}$ there exists $x_*\in B(x_0,\frac{\eps}{2})$ with $\delta_*:= u_0(x_*)>0$.

Then, using \eqref{inclusions2} implies that
\begin{equation*}
  x_*\in \{u_0\geq\delta_*\} \subset \{u(\cdot,t_k)>0\}
  \quad\text{ for all }k\gg 1\;,
\end{equation*}
a contradiction to \eqref{eq:notHausd2}.
\end{proof}

A posteriori we even obtain a quantitative growth bound for the support of solutions.
\begin{corollary}\cite{BF76,EK79}
\label{cor:growthbound}
Let the assumptions of Theorem \ref{1thm:main1} hold.
Then
\begin{equation}
  \{u(\cdot,t)>0\} \subset (\{u_0>0\})_{C_2\sqrt{t|\log t|}}
  \label{eq:growthbound}
\end{equation}
for all $t>0$ sufficiently small.
\end{corollary}
\begin{proof}
By \eqref{eq:nondegeneracy1-lambda}, Theorem \ref{1thm:main1} and the right-continuity of \eqref{eq:oblam2}, \eqref{eq:oblam3} we know that there exist $\kappa>0$ and $t_0>0$ such that
\begin{equation*}
  f(x,t) :=  -1+\frac{g(x)}{\lambda(t)} \leq -\kappa
\end{equation*}
for all $(x,t)$ in an open neighborhood of $\{u=0\}$ in $\Gamma\times [0,t_0]$.

Then we can adopt the proof of \cite[Theorem 3.2]{EK79} to obtain \eqref{eq:growthbound}.
\end{proof}

\section{Initial jump for degenerate data if the first nondegeneracy condition is violated}\label{Jump Section}

So far, we have proven the continuity of the functions $\alpha$ and $\lambda$ in $t=0$ assuming \eqref{eq:nondegeneracy1-lambda} and \eqref{eq:nondegeneracy2}.

This section highlights the necessity of  \eqref{eq:nondegeneracy1-lambda} for the continuity of the function $\lambda$ and of the support of the solution $u$.
More precisely, we will show that if \eqref{eq:nondegeneracy1-lambda} fails, then one cannot expect continuity of the function $\lambda$ nor of the set $\{u(\cdot,t)>0\}$ at $t=0$.
To this end, we assume that \eqref{eq:nondegeneracy2} holds true, while  \eqref{eq:nondegeneracy1-lambda} is violated in the sense that
\begin{equation}\label{violation1}
  |\{u_0=0\} \cap \{(1-g)-\alpha_0 g< 0\}| >0\;,
\end{equation}
where $\alpha_0=\frac{\int_{\{u_0>0\}} (1-g)\,dS}{\int_{\{u_0>0\}} g\,dS}\;$.
By \eqref{alphaAverage}, this assumption is equivalent to
\begin{equation}\label{violation2}
  |\{u_0=0\} \cap \{g> \lambda_0 \}| >0\;,
  \text{ where }\lambda_0=\fint_{\{u_0>0\}} g\;dS.
\end{equation}
The non-generic case in which $|\{u_0=0\} \cap \{g =\lambda_0 \} | >0$ will be addressed below (see Remark \ref{expl3} ).

For the sake of convenience, we choose in the following analysis to study the equivalent problem \eqref{eq:oblam1}-\eqref{eq:oblam4}, see Lemma \ref{2lem:equival}.
\medskip

Our aim is to prove that under assumption \eqref{violation2} the function $\lambda$ and the positivity set $\{u(\cdot,t)>0\}$ both will jump at $t=0$.
It turns out that we can characterize this jump in terms of a variational principle.
More specifically, we define $\Lambda[u_0]$ as follows.
\begin{definition}\label{jump}
For any open set $S\subset \Gamma$, we set
\begin{equation}\label{biglambda defn}
	\Lambda_S:=\sup \bigg \{ \fint_{A} g\;dS\;:\; A\subset\Gamma\text{ measurable with } S\subset A \bigg \}\;.
\end{equation}
In the particular case $S=\{u_0>0\}$, where $u_0:\Gamma\to \R$ denotes the nonnegative continuous initial data, we will write $\Lambda[u_0]:=\Lambda_{\{u_0>0 \}}$ for the sake of simplicity.
\end{definition}

We will prove that  under \eqref{violation2}
\begin{equation*}
  \lim \limits_{t \searrow 0}\lambda(t)=\Lambda[u_0] \quad \text{and} \quad \Lambda[u_0]>\lambda_0\;.
\end{equation*}
Moreover, we will show that the positivity set $\{u(\cdot,t)>0\}$ approximates as $t\to 0^{+}$, one of the sets, for which the maximum in \eqref{biglambda defn} is attained. Notice that we can have several maximizers in \eqref{biglambda defn} differing by sets contained in $\{g = \Lambda_S \}$.

\begin{definition}
For a given nonnegative continuous function $u_0:\Gamma \to \mathbb R$, we set
\begin{equation}
	A^0_*:=\big\{g\geq \Lambda[u_0] \big\} \cup \{u_0>0\}\;.
	\label{jumpset defn1}
	\end{equation}
%where $\Lambda[\cdot]$ is as given in Definition \ref{jump}.
\end{definition}
We will prove later that the maximum in \eqref{biglambda defn} is attained by the set $A^0_*$.

\medskip
Our main goal in this section is to prove the following result.

\begin{theorem}\label{main.p3}
Suppose that \eqref{eq:nondegeneracy2} holds true and that $g$ satisfies \eqref{gassumptions}.
For any $\eta>0$ there exists $\bar t=\bar t (\eta)>0$ such that the positivity set $\{u(\cdot,t)>0\}$ satisfies for all $0<t \leq \bar t(\eta)$
\begin{equation}
  \big(\{u_0>0\} \cup \{g>\Lambda[u_0]\} \big)_{-\eta}
  \subset \{u(\cdot,t)>0\}
  \subset  \big(\{u_0>0\} \cup \{g\geq \Lambda[u_0]\} \big)_{+\eta}. \label{newpositivityset}
\end{equation}
Furthermore,
\begin{equation}
  |\lambda(t)-\Lambda[u_0]| \leq \eta\, \quad \text{for all }0 < t \leq \bar t(\eta)\; .
  \label{initial jump}
\end{equation}
In particular, $\lambda(t) \to \Lambda[u_0]$ as $t \searrow 0$.
\end{theorem}

\medskip

\begin{remark}\label{Add1}
We stress that the inclusions in \eqref{newpositivityset} imply that there exists a set $B(t) \subset \{g=\Lambda[u_0] \}$ such that $\{u(\cdot,t)>0 \} \cup B(t)\to A^0_*$ with respect to the $L^1-$convergence of sets. It is worth noticing that $B(t)$ could in principle be oscillatory.
\end{remark}

\begin{remark}\label{Add2}
It follows from the proof of Theorem \ref{main.p3} and more specifically from \eqref{eq:inclusion-1} that
\begin{equation*}
  \{u_0>0\} \cup \{g>\Lambda[u_0]\}
  \subset \{\liminf_{t\searrow 0}u(\cdot,t)>0\}
  \subset \liminf_{t\searrow 0}\{u(\cdot,t)>0\}
  =\bigcup_{\delta>0}\bigcap_{0<t<\delta}\{u(\cdot,t)>0\}
\end{equation*}
and from \eqref{newpositivityset} that
\begin{equation*}
  \limsup_{t\searrow 0}\{u(\cdot,t)>0\}
  =\bigcap_{\delta>0}\bigcup_{0<t<\delta}\{u(\cdot,t)>0\}
  \subset  \overline{\{u_0>0\}} \cup \{g\geq \Lambda[u_0]\}.
\end{equation*}
Note that $|\overline{\{u_0>0\}}\setminus{\{u_0>0\}}|=0$ by the regularity of the set ${\{u_0>0\}}$.
On the other hand, if $|\{g=\Lambda[u_0]\}|>0$, then the upper and lower inclusion may differ by a set of positive measure.
\end{remark}

\begin{corollary}\label{cor:expl2}
Assume in addition that \eqref{violation2} holds.
Then
\begin{equation*}
  \lim_{t\searrow 0}\lambda(t)=\Lambda[u_0]>\lambda_0 \quad \text{and}\quad \vert A_*^0\setminus \{u_0>0\} \vert>0\;.
\end{equation*}
\end{corollary}

\begin{proof}
The first equality follows from Theorem \ref{main.p3}.

Let us consider the set $A=\{u_0>0\} \cup B$, where $B:=\{u_0=0\} \cap \{g>\lambda_0\}$ satisfies $\vert B \vert >0$ by \eqref{violation2}.
We obtain that
\begin{align*}
  \fint_A g\;dS=\frac{1}{\vert A \vert } \bigg(\int_{\{u_0>0\}} g\;dS+\int_{B} g\;dS\bigg)
  > \frac{1}{\vert A \vert } \big(\lambda_0\vert \{u_0>0\}\vert+\lambda_0\vert B \vert \big)=\lambda_0\;,
\end{align*}
which yields due to Definition \ref{jump} that $\Lambda[u_0]>\lambda_0$.

Moreover,
\begin{equation*}
  \big|A_*^0\setminus \{u_0>0\}\big|=\big|\{g\geq \Lambda[u_0]\}\cap\{u_0=0\}\big| >0,
\end{equation*}otherwise using once again Definition \ref{jump} we obtain that $\Lambda[u_0]=\lambda_0$ which contradicts the first statement of the Corollary.
\end{proof}

\begin{remark}\label{expl1}
If in addition to the assumptions in Theorem \ref{main.p3} also \eqref{eq:nondegeneracy1} holds, Theorem \ref{main.p3} reduces to Theorem \ref{3thm:main1}.
Indeed, \eqref{eq:nondegeneracy1} is by \eqref{alphaAverage} equivalent to
\begin{equation}\label{R1}
  g<\lambda_0\quad \text{ in } \{u_0=0\}\;.
\end{equation}
Then, for any set $A$ as in Definition \ref{jump} with $\vert A\setminus \{u_0>0\} \vert >0$, we obtain due to \eqref{R1} and \eqref{eq:oblam4} that
\begin{align*}
  \fint_A g\;dS=\frac{1}{\vert A \vert } \bigg(\int_{\{u_0>0\}} g\;dS+\int_{A\setminus \{u_0>0\}} g\;dS\bigg)
  <\frac{1}{\vert A \vert } \big(\lambda_0\vert \{u_0>0\}\vert+\lambda_0\vert A\setminus \{u_0>0\} \vert \big)=\lambda_0\;.
\end{align*}
Thus, since $\lambda_0 \leq \Lambda$ by definition of $\lambda_0$, we conclude that $\Lambda[u_0]=\lambda_0$ and \eqref{initial jump} reduces to \eqref{continuityy}.
Furthermore, by \eqref{R1} and the fact that $\Lambda[u_0]=\lambda_0$, we observe that the right- as well as the left-hand side of \eqref{newpositivityset} can be written now equivalently as
\begin{equation*}
	\big\{g\geq \Lambda[u_0] \big\} \cup \{u_0>0\}=\big\{g\geq \lambda_0 \big\} \cup \{u_0>0\}=\{u_0>0\}=\big\{g> \Lambda[u_0] \big\} \cup \{u_0>0\}\;.
\end{equation*}
Hence, \eqref{newpositivityset} reduces to \eqref{inclusions1} and therefore, Theorem \ref{main.p3} can be considered as a generalization of Theorem \ref{3thm:main1}.
\end{remark}

\begin{remark}\label{expl3}
We will not consider in detail the non-generic case
\begin{equation}\label{non-generic}
  \vert \{u_0=0\} \cap \{g= \lambda_0 \} \vert >0\;.
\end{equation}
It is worth noticing though, that in this case one can deduce, arguing in a similar way as in Remark \ref{expl1}, that $\lambda_0=\Lambda[u_0]$ while $\vert \{u_0>0\}\Delta A_*^0 \vert>0$.  Hence, it is possible to choose initial data that satisfy \eqref{non-generic} such that $\lambda$ will be continuous at $t=0$, while  the interface will jump.
\end{remark}

\begin{remark}\label{expl4} In the case of plateaus
	\begin{equation}\label{plateau}
	\vert \{g= \Lambda[u_0] \} \vert>0\;,
	\end{equation}in which $g$ takes a constant value, it is not possible to decide using only the assumptions of Theorem \ref{main.p3} if the points of the set $\{g= \Lambda[u_0] \}$ are contained in the positivity set $\{u(\cdot,t)>0\}$ or in the set $ \{u(\cdot,t)=0\}$ as $t\to 0^+$.

	The reason for that is that in this case the sign of $\lambda(t)-\Lambda[u_0]$ could depend on the details of the initial data $u_0$ and as a consequence, the points of $\{g= \Lambda[u_0] \}$ could lie either in the positivity set $\{u(\cdot,t)>0\}$ or in the set $ \{u(\cdot,t)=0\}$ as $t\to 0^+$. As a matter of fact, a similar situation can occur not only when \eqref{plateau} holds, but also in any set $A\subset \{g= \Lambda[u_0] \}$ such that $A_{+\delta} \subset \{g\leq  \Lambda[u_0] \}$ for any $\delta>0$.

\end{remark}

\subsection{A variational characterization of $\Lambda$}

We collect some properties of $\Lambda[u_0]$ and $A^0_*$ that will be used in the following analysis.

\begin{lemma}\label{well-posed}
\begin{enumerate}[label=(\roman*)]
	\item The maximum in \eqref{biglambda defn} is attained by the set $A^0_*$.
    Any maximizer is (up to sets of measure zero) contained in $A^0_*$.
    If $|\{g=\Lambda[u_0]\}\cap\{u_0=0\}|=0$ then $A^0_*$ is unique (up to sets of measure zero).
	\item For every $x\in {(A^0_*)}^{\mathsf{c}}$ it holds that $g(x)< \Lambda[u_0]$.
	\item For every $x\in \partial A^0_* \cap \interior{\{u_0=0\}}$ it holds that $g(x)=\Lambda[u_0]$.
\end{enumerate}
\end{lemma}

\begin{proof}
We first prove $\textit{(i)}$.
For any $A,B\subset\Gamma$ we observe that
\begin{align}
  \fint_{A\cup B} \big(\Lambda[u_0]-g\big)\,dS &= \frac{|A|}{|A\cup B|}\fint_{A} \big(\Lambda[u_0]-g\big)\,dS \nonumber\\
  &\qquad + \frac{|B|}{|A\cup B|}\fint_{B} \big(\Lambda[u_0]-g\big)\,dS - \frac{|A\cap B|}{|A\cup B|}\fint_{A\cap B} \big(\Lambda[u_0]-g\big)\,dS.
  \label{note}
\end{align}
For $A,B\subset\Gamma$ with $\{u_0>0\}\subset A\cap B$ \eqref{note} equation combined with the definition of $\Lambda[u_0]$ implies
\begin{align*}
  \Lambda[u_0]-\fint_{A\cup B} g\,dS &\leq \frac{|A|}{|A\cup B|}\Big(\Lambda[u_0]-\fint_{A} g\,dS\Big) + \frac{|B|}{|A\cup B|}\Big(\Lambda[u_0]-\fint_{B} g\,dS\Big)\\
  &\leq \Big(\Lambda[u_0]-\fint_{A} g\,dS\Big) + \Big(\Lambda[u_0]-\fint_{B} g\,dS\Big).
\end{align*}

We next show that the maximum in \eqref{biglambda defn} is attained.
We therefore consider a sequence $(A_k)_k$ of measurable subsets of $\Gamma$ with $\{u_0>0\}\subset A_k$ and
\begin{equation*}
  \fint_{A_k} g\,dS \geq \Lambda[u_0]-2^{-k-1}
\end{equation*}
for all $k\in\N$.
Let $A^N:=\bigcup_{k\geq N}$, then the first item implies that
\begin{align*}
  \Lambda[u_0]-\fint_{A^N} g\,dS &\leq \sum_{k\geq N}\Big(\Lambda[u_0]-\fint_{A_k} g\,dS\Big)\leq 2^{-N}\,\to\, 0 \quad \text{as }\;N\to\infty.
\end{align*}
Moreover $(A^N)_{N\in\N}$ is monotonically decreasing and converges to $A^\infty=\limsup_{k\to\infty}A_k$.
By monotonicity we conclude
\begin{align*}
  \fint_{A^\infty} g\,dS &= \lim_{N\to\infty}\fint_{A^N} g\,dS \geq \Lambda[u_0].
\end{align*}
On the other hand $\{u_0>0\}\subset A^\infty$, hence $A^\infty$ is a maximizer in \eqref{biglambda defn}.
We next show that $A^0_*$ is also a maximizer.
Let $A^\infty$ be an arbitrary maximizer, which exists by the previous item.
We first use $A^\infty=(A^\infty\cap A^0_*)\cup (A^\infty\setminus A^0_*)$ and deduce by the first item
\begin{align*}
  \Lambda[u_0] &= \frac{|A^\infty\cap A^0_*|}{|A^\infty|}\fint_{A^\infty\cap A^0_*} g\,dS + \frac{|A^\infty\setminus A^0_*|}{|A^\infty|}\fint_{A^\infty\setminus A^0_*} g\,dS\\
  &\qquad\leq \frac{|A^\infty\cap A^0_*|}{|A^\infty|}\Lambda[u_0] + \frac{|A^\infty\setminus A^0_*|}{|A^\infty|}\Lambda[u_0] = \Lambda[u_0],
\end{align*}
where the inequality follows by the definition of $\Lambda[u_0]$, and since $g<\Lambda$ on $A^\infty\setminus A^0_*\subset\{u_0=0\}$.
This shows that the inequality needs to be an equality, which implies $|A^\infty\setminus A^0_*|=0$ and $A^\infty\subset A^0_*$ up to a set of measure zero.

On the other hand, we use $A^0_*=A^\infty\cup (A^0_*\setminus A^\infty)$ up to a set of measure zero and obtain by $\{u_0>0\}\subset A^0_*$ and the definition of $\Lambda[u_0]$ together with the first item
\begin{align*}
  \Lambda[u_0]\geq \fint_{A^0_*} g\,dS &= \frac{|A^\infty|}{|A^0_*|}\fint_{A^\infty} g\,dS +\frac{|A^0_*\setminus A^\infty|}{|A^0_*|}\fint_{A^0_*\setminus A^\infty} g\,dS \\
  &\geq \frac{|A^\infty|}{|A^0_*|}\Lambda[u_0]+\frac{|A^0_*\setminus A^\infty|}{|A^0_*|}\Lambda[u_0]=\Lambda[u_0],
\end{align*}
where we have used the second item and $g\geq\Lambda$ on $A^0_*\setminus A^\infty\subset\{u_0=0\}$.
Equality in the inequalities above requires $\fint_{A^0_*} g\,dS=\Lambda[u_0]$ and $g=\Lambda[u_0]$ almost everywhere on $A^0_*\setminus A^\infty$.
This proves the first item of the lemma.

\medskip
The items $\textit{(ii)}$ and $\textit{(iii)}$ follow from the definition of the set $A^0_*$.
\end{proof}

We collect some further properties of the functional $\Lambda$.
\begin{proposition}\leavevmode
\label{Lambda_properties1}
\begin{enumerate}
 \item Monotonicity: For any measurable sets $S_1, S_2$ with $S_1\subset S_2$ we have $\Lambda_{S_1}\geq \Lambda_{S_2}$.
The maximizer
 \begin{equation*}
  S_*^j := S_j\cup \{g\geq \Lambda_{S_j}\},\,j=1,2
 \end{equation*}
 (see Lemma \ref{well-posed}) satisfies $S_*^1\subset S_*^2$.
 \item Continuity: For any open set $S\subset\Gamma$, let the sets $S_{-\delta}$ be as in Definition \ref{app delta sets defn}.
Then
 $\Lambda_{S_{-\delta}}\searrow \Lambda_S$ as $\delta\searrow 0$.
\end{enumerate}
\end{proposition}

\begin{proof}
\begin{enumerate}
\item  The property $\Lambda_{S_1}\geq \Lambda_{S_2}$ follows directly from the definition of $\Lambda_S$.
This also implies $\{g\geq \Lambda_{S_1}\}\subset \{g\geq\Lambda_{S_2}\}$, hence $S_*^1\subset S_*^2$ by definition.
\item  By definition $S_{-\delta_1}\supset S_{-\delta_2}$ for $\delta_1<\delta_2$ and hence the first item implies that $\delta\mapsto\Lambda_{S_{-\delta}}$ is increasing.
Therefore
 \begin{equation*}
  \lambda_0 := \lim_{\delta\searrow 0}\Lambda_{S_{-\delta}}
 \end{equation*}
 exists.
By the first item the maximizers
 \begin{equation*}
  S_*^{-\delta }:= S_{-\delta}\cup \{g\geq \Lambda_{S_{-\delta}}\}
 \end{equation*}
 are monotone in the sense that $S_*^{-\delta_1}\supset S_*^{-\delta_2}$ for $0<\delta_1<\delta_2$.
 We deduce that
 $\Chi_{S_*^{-\delta }}$ is monotonically increasing with $\delta\searrow 0$ and hence converges to $ \Chi_{S_0}$ with $S_0:= \bigcup_{\delta>0} S_*^{-\delta}$, in particular
 \begin{equation*}
  \Chi_{S_*^{-\delta }}\to \Chi_{S_0}\quad\text{ in }L^1(\Gamma)
 \end{equation*}
 and therefore
 \begin{equation*}
  \fint_{S_0}g = \lim_{\delta\searrow 0} \Lambda_{S_{-\delta}}.
 \end{equation*}
 Since $S$ is open and $g$ is continuous we also have
 \begin{equation*}
  S_0 = \bigcup_{\delta>0} S_*^{-\delta} = \bigcup_{\delta>0} S_{-\delta} \cup \bigcup_{\delta>0} \{g\geq \Lambda_{S_{-\delta}}\}
  = S \cup \{g>\lim_{\delta\searrow 0} \Lambda_{S_{-\delta}}\},
 \end{equation*}
 hence
 \begin{equation*}
  S\subset S_0
 \end{equation*}
 which implies by \eqref{biglambda defn}
 \begin{equation*}
  \Lambda_S \geq \fint_{S_0}g = \lim_{\delta\searrow 0} \Lambda_{S_{-\delta}}.
 \end{equation*}
 On the other hand $S_{-\delta}\subset S$, hence using once again the monotonicity in the first item, we obtain
 \begin{equation*}
  \lim_{\delta\searrow 0} \Lambda_{S_{-\delta}} \geq \Lambda_S,
 \end{equation*}
 which proves equality.
\end{enumerate}
\end{proof}

We next consider a solution $u$ of \eqref{eq:obalph1}-\eqref{eq:obalph4} and connect the functional $\Lambda$ to the function $\lambda$.
\begin{corollary}\label{Lambda properties2}
Recall that $\Lambda[u(\cdot,t)]=\Lambda_{\{u(\cdot,t)>0\}}$.
Then, it holds
 \begin{align}
  \lambda(t) &\leq \Lambda[u(\cdot,t)] \quad\text{ for all }t\in (0,T),
  \label{eq:lambdas-ineq}\\
  \lambda(t) &= \Lambda[u(\cdot,t)] \quad\text{ for almost all }t\in (0,T).
 \label{eq:lambdas}
\end{align}
\end{corollary}
\begin{proof}
The inequality \eqref{eq:lambdas-ineq} follows from the definitions of $\Lambda$ and $\lambda$.

Let $A_*(t):=\{u(\cdot,t)>0\}\cup \{g\geq \Lambda[u(\cdot,t)]\}$.
By \eqref{eq:oblam2} for all $t\in (0,T)\setminus N$, $|N|=0$ we have $g\leq \lambda(t)$ in $\{u(\cdot,t)=0\}$.
Assume for some $t\in (0,T)\setminus N$ that $\lambda(t)<\Lambda[u(\cdot,t)]$.
Then
\begin{align*}
  A_*(t)\setminus \{u(\cdot,t)>0\}=\{g\geq \Lambda[u(\cdot,t)]\}\cap\{u(\cdot,t)=0\}=\emptyset,
\end{align*}
which shows $A_*(t)=\{u(\cdot,t)>0\}$ and
\begin{align*}
  \Lambda[u(\cdot,t)]= \fint_{A_*(t)}g
  &= \fint_{\{u(\cdot,t)>0\}}g=\lambda(t).
\end{align*}
\end{proof}

\subsection{Proofs of Theorems \ref{main.p3} and \ref{1thm:main2}} \label{NMS}

Throughout this section, let $u$ be a solution to \eqref{eq:oblamalt1}- \eqref{eq:oblamalt4}.
We also recall that we only assume \eqref{eq:nondegeneracy2} but not \eqref{eq:nondegeneracy1}.

% As we have already mentioned in Remark \ref{cor:expl2}, if \eqref{eq:nondegeneracy1} fails in a set of positive measure, or equivalently \eqref{violation2} is valid, one cannot expect any continuity properties for the positivity set $\{u(\cdot,t)>0\}$ nor for the function $\lambda$.
% Our goal is to derive precise estimates for the corresponding jumps for short times.

\medskip
Motivated by Remark \ref{expl1}, our strategy is to approximate $u$ by a solution to \eqref{eq:oblam1}-\eqref{eq:oblam4} with suitably modified initial data $u_n^0$.
The latter are chosen such that they, in particular, converge uniformly to $u_0$ but such their support, on the other hand, approximates $A_*^0$, i.e.
\begin{equation*}
  u_n^0 \searrow u_0 \quad\text{ uniformly on }\Gamma,\quad
  A_*^0 =\bigcap_{n\in\N} \{u_n^0>0\}.
\end{equation*}
A key property of the modified solutions will be that we can apply the continuity results obtained in Section \ref{Continuity Section}.

\medskip
We first fix a suitable family of initial data $u^0_n$ and describe specific properties that we can obtain.
We denote in the following
\begin{equation*}
  \lambda^0_n := \fint_{\{u^0_n>0\}}g\,dS.
\end{equation*}

\begin{lemma}\label{eta data regular}
Let $g \in C^2(\Gamma)$ and assume that the set $\{u_0>0\}$ is regular according to \eqref{eq:nondegeneracy2}.
There exists a non increasing sequence $(\gamma_n)_n$ of numbers with
\begin{equation}\label{etaseq properties}
  \gamma_n \searrow 0 \;\; \text{as } n\to \infty
\end{equation}
and a sequence $(u_n^0)_{n\in\N}$ of nonnegative functions in $C^2(\Gamma)$ such that the following properties hold for all $n\in\N$:
\begin{enumerate}[label={(\arabic*)}]
  \item $u_0 \leq u^0_{n+1} \leq u^0_n$. \label{first}
  \item $\{u^0_n>0\} \supset \{ u^0_{n+1}>0 \}$. \label{second}
  \item $u^0_n>0$ in $\{g\geq \Lambda[u_0]-{\gamma_n}\}$. \label{third}
  \item $u^0_n=0$ in $\{g \leq \Lambda[u_0]-2{\gamma_n}\} \cap \{u_0=0\}$. \label{fourth}
  \item The set $\{u^0_n>0 \}$ is regular.
  \label{sixth}
  \item $\displaystyle
  \vert \lambda_n^0-\Lambda[u_0] \vert < \frac{{\gamma_n}}{4}$.
  \label{eighth}
  \item $\|u^0_n-u_0\|_{C^0(\Gamma)}\leq \gamma_n$.
  \label{nine}
\end{enumerate}
and such that
\begin{enumerate}[resume,label={(\arabic*)}]
  \item $\big| \{u^0_n>0 \} \setminus A_*^0  \big| \to 0\; \quad \text{as } n \to \infty$ \label{fifth}
\end{enumerate}
Moreover, for any $\eta>0$ there exists $n^*=n^*(\eta)$ such that
\begin{enumerate}[resume,label={(\arabic*)}]
  \item
  $\displaystyle \{ u^0_n>0\} \subset \big (A_*^0 \big )_{+\eta}$\quad for any $n \geq n^*$.
  \label{seventh}
\end{enumerate}
\end{lemma}

We prove this lemma in Appendix \ref{A:initial_data}. Fixing sequences $(\gamma_n)_n$ and $(u_n^0)$ as in Lemma \ref{eta data regular} we now define our main approximation.

\medskip
{\bf Regularization:} Let $(u_n)_n$ be the unique solution of the problem
\begin{align}
  \partial_t u_n -\Delta u_n&= -\bigg(1-\frac{g}{\lambda_n(t)} \bigg )H(u_n), \quad &&\text{in } \Gamma_T \label{eta eq}\\
  g&\leq \lambda_n, \quad &&\text{a.e in } \{u_n=0\} \label{eta key step}
  \\
  u_n(\cdot,0)&=u^0_n &&\text{on } \Gamma, \label{eta data}
\end{align}
where
\begin{equation*}
  \lambda_n(t)=\fint_{\{u_n(\cdot,t)>0\}} g\;dS\;.
\end{equation*}
By \eqref{contraction} and items \ref{first}, \ref{nine} of Lemma \ref{eta data regular} we deduce that for all $n_1\geq n_2$ it holds
\begin{equation}
  u(\cdot,t)\leq u_{n_1}(\cdot,t) \leq u_{n_2}(\cdot,t)\quad\text{ for all }t\geq 0,
  \label{eq:un-monoton}
\end{equation}
and by \cite[Theorem 3.1]{LNRV21}
\begin{equation}
  u_n\to u \quad\text{ in }C^0([0,T],L^1(\Gamma))\text{ for all }T>0\;.
  \label{eq:conv-un}
\end{equation}
Now we proceed to the proof of the main result of this Section, Theorem \ref{main.p3}.

\begin{proof}[Proof of Theorem \ref{main.p3}]
As we already mentioned, our aim is to approximate \eqref{eq:oblam1}-\eqref{eq:oblam4} by the regularized problem  \eqref{eta eq}-\eqref{eta data} and then pass to the limit as $n\to \infty$.
The strategy of the proof consists of six steps.
More specifically, in the first five steps we derive upper and lower estimates for the sequence of functions $\lambda_n$ and the positivity sets $\{u_n(\cdot,t)>0\}$.
In the sixth step we consider the limit $n \to \infty$.

\medskip
{\bf Step 1:} First we will prove that there exists a modulus of continuity $\hat\delta_1$ and for any $r>0$ a $t_1=t_1(r)>0$ such that
\begin{align}
  \lambda_n(t) &\leq \Lambda \big [u_0 ]+\hat \delta_1(r)\;,
  \label{initial jumpreg}\\
  \{u_0>r \} &\subset \Big\{u_n(\cdot,t)>\frac{r}{2} \Big\}, \label{lower est noeta1}
\end{align}
for all $n\in \mathbb N$ and for all $0\leq t \leq t_1(r)$.

\medskip
We will derive the uniform inclusion property \eqref{lower est noeta1} by means of a suitable subsolution.
To this end, we define $\tilde u(x,t):=S(t)u_0(x)-t$ where $S(t)$ denotes the heat semigroup on $\Gamma$ and we calculate
\begin{equation*}
  \partial_t \tilde u -\Delta \tilde u=-1< -\bigg(1-\frac{g}{\lambda_n(t)} \bigg )H(u_n)=\partial_t u_n -\Delta u_n
\end{equation*}
in $\Gamma_T$.
Furthermore, $ \tilde u(\cdot,0)=u_0\leq u_n(\cdot,0)$ on $\Gamma$ by Lemma \ref{eta data regular}, item \ref{first}.
Hence, we obtain by a comparison principle argument that
\begin{equation}\label{voitheia}
  u_n(\cdot,t) \geq S(t)u_0-t \quad \text{ in } \Gamma_T\;.
\end{equation}
Now, since $S(t)u_0-t \to u_0$ as $t\to 0^+$, it follows that for all $r>0$ there exists $t_1(r)>0$ such that
\begin{equation*}
  u_n(\cdot,t) > \frac{r}{2}>0 \quad \text{in } \{u_0>r \}\quad\text{ for all }n\in\N\text{ and } 0\leq t \leq t_1(r)\;,
\end{equation*}
which proves \eqref{lower est noeta1}.
Next, item $(1)$ of Proposition \ref{Lambda_properties1} and \eqref{eq:lambdas-ineq} in  Corollary \ref{Lambda properties2} yield
\begin{equation}
  \Lambda \big [\{u_0>r \} \big ] \geq \Lambda \big [\{u_n(\cdot,t)>0 \} \big ]
  \geq \lambda_n(t).
  \label{eq:4.25a}
\end{equation}
By the second item of Proposition \ref{Lambda_properties1} and Lemma \ref{-delta set}, we conclude that \eqref{initial jumpreg} holds.

\medskip
{\bf Step 2:} Let $t_1$ be as in Step 1.
For any $\sigma>0$ there exist $r_1(\sigma)>0$, $0<t_2(\sigma)\leq t_1(r_1(\sigma))$ and a positive function $\omega^*_2:[0,\infty)^2\to\R^+$ such that for all $t \leq t_2(\sigma)$ we have
\begin{equation}
  \overline{\{g>\Lambda[u_0]+\sigma\}} \subset \big\{u_n(\cdot,t)>\omega^*_2(\sigma,t)\big\}\quad
  \text{ for all }n\in\N.
  \label{eq:step2a}
\end{equation}

\medskip
Consider the set $A_\sigma:= \{g>\Lambda[u_0]+\gamma\sigma\}$, where $0<\gamma<1$ is chosen such that $\Lambda[u_0]+\gamma\sigma$ is a regular value of $g$.
Then $A_\sigma$ has a $C^2$-regular boundary and it holds
\begin{equation*}
   A_\sigma\ssubset \{g>\Lambda[u_0]+\sigma\}.
\end{equation*}
By \eqref{initial jumpreg}, we calculate that in $A_\sigma \times [0,t_1(r)]$ it holds
\begin{align*}
  \partial_t u_n -\Delta u_n
  = -\bigg(1-\frac{g}{\lambda_n(t)} \bigg )H(u_n) &\geq \bigg (-1+\frac{\Lambda[u_0]+\sigma}{\Lambda[u_0]+\hat \delta_1(r)} \bigg )H(u_n)\\
  &=\bigg (\frac{\sigma-\hat \delta_1(r)}{\Lambda[u_0]+ \hat \delta_1(r)} \bigg )H(u_n)\;.
\end{align*}
For sufficiently small $r_1=r_1(\sigma)>0$ and $t_1=t_1(r_1(\sigma))$ from Step 1 we obtain
\begin{align*}
  \partial_t u_n -\Delta u_n \geq \frac{1}{2} \frac{\sigma}{\Lambda[u_0]}H(u_n)\; \quad \text{in }\;  A_\sigma\times [0, t_1].
\end{align*}
Due to the fact that $u_n^0>0$ in $\{g\leq\Lambda[u_0]\}$ by Lemma \ref{eta data regular}, item \ref{third}, there exists a maximal $0<\hat t_n\leq t_1$ such that $u_n>0$ in $ A_\sigma\times [0, \hat t_n)$.
Therefore, we find that
\begin{align*}
  \partial_t u_n -\Delta u_n \geq \frac{1}{2} \frac{\sigma}{\Lambda[u_0]}\; \quad \text{in }\;  A_\sigma\times [0, \hat t_n).
\end{align*}

We construct a suitable subsolution of $u_n$.
To this end, we let $U$ denote the solution of
\begin{align*}
  \partial_t U -\Delta U&= \frac{1}{2} \frac{\sigma}{\Lambda[u_0]}\; \quad &&\text{in } \;A_\sigma \times (0,\hat t_n]\;,\\
  U&=0\; \quad &&\text{in }\; \partial A_\sigma \times (0,\hat t_n]\;,\\
  U(\cdot,0)&=0\; \quad &&\text{in }\; A_\sigma\;.
\end{align*}
It follows immediately that $U>0$ in $ A_\sigma \times [0,\hat t_n]$.
Due to Lemma \ref{eta data regular}, item \ref{third}, we obtain that $u^0_n> U(\cdot,0)$ in $\overline{  A_\sigma}$ and moreover that $u_n \geq U$ on $\partial   A_\sigma \times [0,\hat t_n]$.
A comparison argument yields then that $u_n\geq U$ in $\overline{ A_\sigma} \times [0,\hat t_n]$.

\medskip
Now assume that $\hat t_n<t_1$. Then $\min_{\overline{A_\sigma}}u_n(\cdot,\hat t_n)=0$.
Since $u_n(\cdot,\hat t_n)\geq U(\cdot,\hat t_n)>0$ in $A_\sigma$  there exists a boundary point $\hat x \in \partial A_\sigma$ with
\begin{equation*}
  u_n(\hat x,\hat t_n)=0=U(\hat x,\hat t_n).
\end{equation*}
However, using the parabolic Hopf lemma, $U(\cdot,\hat t_n)>0$ in $A_\sigma$ and $0=\min_\Gamma u_n$ we deduce that
\begin{equation*}
  0 < \nabla \big (U(\hat x,\hat t_n)-u_n(\hat x,\hat t_n)\big) \cdot \nu \leq 0,
\end{equation*}
which is a contradiction.

Therefore, $u_n\geq U$ in  $\overline {A_\sigma} \times [0,t_1]$ and in particular we infer a uniform lower estimate for $u_n$, that is
\begin{equation}
  u_n(\cdot,t) \geq \omega^*_2(\sigma,t)
  \quad\text{ in }\;\overline{\{g>\Lambda[u_0]+\sigma\}} \times (0,t_1]
  \quad\text{ for all }n\in\N\;,
  \label{UnLB2}
\end{equation}
where
\begin{equation*}
  \omega^*_2(\sigma,t) := \inf_{\{g>\Lambda[u_0]+\sigma\}}U(\cdot,t)>0,\quad t\in  (0,t_1]
\end{equation*}
is positive by \eqref{eq:4.25a} and $U>0$ in $A_\sigma\times [0,\hat{t}_n]$.

This yields \eqref{eq:step2a} and finishes Step 2 of the proof.

\medskip
{\bf Step 3:} Next we will show that for any $\eta>0$ there exists $t_3(\eta)$ such that
\begin{equation}
  \Big(\{u_0>0\} \cup \{g>\Lambda[u_0]\} \Big)_{-\eta}  \subset \{u_n(\cdot,t)>0\}\; \quad \text{for all }\; 0\leq t \leq t_3(\eta) \label{inclusions reg}
\end{equation}
and for all $n\in\N$.

In fact, for $\sigma>0$ and $r=r_1(\sigma)$ we have by \eqref{lower est noeta1} and \eqref{eq:step2a} that
\begin{align}
  \big\{u_0>r_1(\sigma)\big\}\cup \big\{g>\Lambda[u_0]+\sigma\big\}
  &= \big\{ (u_0-r_1(\sigma))_++(g-\Lambda[u_0]-\sigma)_+>0\big\}\nonumber\\
  &\subset \{u_n(\cdot,t)>0\}\quad\text{ for all }0\leq t\leq t_2(\sigma).
  \label{eq:step2b}
\end{align}
On the other hand, Lemma \ref{-delta set} yields for some modulus of continuity $\delta$ for any $0\leq t\leq t_2(\sigma)$
\begin{align}
  \Big(\{u_0>0\} \cup \{g>\Lambda[u_0]\}\Big)_{-\eta}
  &=\Big(\{u_0+(g-\Lambda[u_0])_+>0\}\Big)_{-\eta}\nonumber\\
  &\subset \big\{u_0+(g-\Lambda[u_0])_+>\delta(\eta)\big\}\nonumber\\
  &\subset \big(\big\{u_0>\delta(\eta)/2\}\cup  \{g>\Lambda[u_0]+\delta(\eta)/2\big\}\big).
  \label{eq:step2c}
\end{align}
Choosing $\sigma=\sigma(\eta)>0$ sufficiently small such that $\sigma,r_1(\sigma)<\delta(\eta)/2$ and setting $t_3(\eta)=t_2(\sigma)$ we deduce from \eqref{eq:step2b}, \eqref{eq:step2c} the inclusion property \eqref{inclusions reg}.

\medskip
{\bf Step 4:} Following Step 3, we will show that for any $\eta>0$ there exists $t_4(\eta)$ such that, with $n^*=n^*(\eta)$ as in item \ref{seventh} of Lemma \ref{eta data regular}, for all $n\geq n^*(\eta)$
\begin{equation}
  \{u_n(\cdot,t)>0\}\subset
  \Big(\{u_0>0\} \cup \{g\geq \Lambda[u_0]\}\Big )_{+2\eta}\;
  \quad \text{ for all }t \leq t_4(\eta)\;. \label{inclusions reg'}
\end{equation}

\medskip
To prove this claim let $\eta>0$ be given and let $n^*=n^*(\eta)$ be as in item \ref{seventh} of Lemma \ref{eta data regular}.
Due to the items \ref{first} and \ref{seventh} of Lemma \ref{eta data regular} we obtain that
\begin{equation}\label{n.inclusions1sides}
  \begin{aligned}
    &\big(\{u_0>0 \}\big)_{-\eta}\subset \big(\{u^0_{n^*}>0 \}\big)_{-\eta} \notag \\
    &\big( \{u^0_{n^*}>0\} \big )_{+\eta}\subset \big (\{u_0>0 \} \cup \{g \geq \Lambda[u_0] \}\big)_{+ 2\eta} \notag.
  \end{aligned}
\end{equation}
By the monotonicity property \eqref{eq:un-monoton} we also have $u_{n}(\cdot,t) \leq u_{n^*}(\cdot,t)$ for any $n\geq n^*$ and for all $t \geq 0$. This in particular implies that for any $n\geq n^*$
\begin{equation}\label{toumpa1}
  \{u_n(\cdot,t)>0\} \subset \{u_{n^*}(\cdot,t)>0\}\quad\text{ for all }t\geq 0.
\end{equation}

\medskip
As mentioned in the beginning of this subsection, we will take advantage of the continuity results in Section \ref{Continuity Section} for both $u_{n^*}$ and $\lambda_{n^*}$ at $t=0$.
In order to apply Theorem \ref{3thm:main1} we need that $\lambda^0_{n^*}$ and the initial data $u_{n^*}^0$ satisfy the conditions \eqref{eq:nondegeneracy1} and \eqref{eq:nondegeneracy2}, that is
\begin{equation}\label{suff cond}
  g<\lambda^0_{n^*}-\theta_{n^*} \;\text{ in } \{u^0_{n^*}=0\}\qquad \text{and} \qquad \text{the set } \; \{u^0_{n^*}>0 \} \;\text{ is regular }
\end{equation}
for some $\theta_{n^*}>0$.
By Lemma \ref{eta data regular}, item \ref{fifth} the second condition is fulfilled.
Moreover by Lemma \ref{eta data regular}, item \ref{third} and item \ref{eighth} we obtain that
\begin{equation}\label{pass}
  g<\Lambda[u_0]-\gamma_{n^*} \leq \lambda^0_{n^*} -\gamma_{n^*}+\big |\lambda^0_{n^*}-\Lambda[u_0]\big |\leq \lambda^0_{n^*}-\frac{3{\gamma_{n^*}}}{4}\quad \text{in } \{u^0_{n^*}=0 \}\;.
\end{equation}
Hence, the first statement in \eqref{suff cond} is satisfied with  $\theta_{n^*}=\frac{3{\gamma_{n^*}}}{4}>0$
We deduce by Theorem \ref{3thm:main1} that there exists $t_4(\eta):=t^*(\eta; \theta_{n^*},u^0_{n^*},g)>0$ such that
\begin{align}
  |\lambda_{n^*}(t)-\lambda^0_{n^*}| &<\eta, \;\label{lambdan* continuity}\\
  \{u_{n^*}(\cdot,t)>0\} &\subset \big(\{u^0_{n^*}>0 \}\big)_{+\eta}
  \label{un* continuity}
\end{align}
for all $t \leq t_4(\eta)$.
Combining this and \eqref{n.inclusions1sides}, \eqref{toumpa1} we conclude that for any $n\geq n^*$ the inclusion \eqref{inclusions reg'} holds.

\medskip
{\bf Step 5:} We justify a uniform lower bound for $\lambda_n$.
More precisely, for any $\eta>0$ and $n^*=n^*(\eta)$, $t_4(\eta)$ as chosen in Step 4 there exists a modulus of continuity $\hat\omega$ independent of $n\geq n^*$ such that for any $n\geq n^*$
\begin{equation}
  \lambda_n(t)\geq \Lambda[u_0] - \hat \omega(\eta) \quad \text{for all } 0\leq t \leq t_4(\eta)\; \label{initial jumpreg'}
\end{equation}
holds.

\medskip
Let us fix in the following an arbitrary $t\in [0,t_4(\eta)]$.
The key idea here is to rewrite $\{u_n(\cdot,t)>0\}$ as a union of disjoint sets.
More precisely, we write
\begin{equation}
  \{u_n(\cdot,t)>0\} = A_n \cup B_n \cup C_n
  \label{suit union}
\end{equation}
with
\begin{align*}
  A_n &:= \{u_n(\cdot,t)>0\}\cap \big(  \{u_0>0\} \cup \{g>\Lambda[u_0] \}  \big),\\
  B_n &:= \{u_n(\cdot,t)>0\}\cap \big(  \{u_0=0\} \cap \{g=\Lambda[u_0] \}  \big),\\
  C_n &:= \{u_n(\cdot,t)>0\}\cap \big(  \{u_0=0\} \cap \{g<\Lambda[u_0] \}  \big).
\end{align*}

Furthermore, by \eqref{inclusions reg} and \eqref{inclusions reg'} we notice that for all $0\leq t \leq t_4(\eta)$
\begin{align}
  A_n &\supset \big(\{u_0>0 \} \cup \{g>\Lambda[u_0] \} \big)_{-\eta}
  \quad \text{for all }\;n\;, \label{b1} \\
  C_n &\subset \big(\{u_0>0 \} \cup \{g \geq \Lambda[u_0] \} \big)_{+2\eta} \cap  \big(  \{u_0=0\} \cap \{g<\Lambda[u_0] \}  \big) \quad \text{for all } \;n\geq n^*. \label{b2}
\end{align}

\medskip
Due to \eqref{suit union} ,\eqref{b1} and the fact that $0<g<1$ we derive for all $n$
\begin{align}
  \lambda_n(t)&=\frac{1}{|A_n|+|B_n|+|C_n|} \bigg ( \int \limits_{A_n}g\;dS+ \int \limits_{B_n}g\;dS+\int \limits_{C_n}g\;dS \bigg) \notag\\
  &\geq \frac{1}{|A_n|+|B_n|+|C_n|} \bigg ( \int \limits_{A_n}g\;dS+ \int \limits_{B_n}g\;dS \bigg) \notag \\
  &\geq \frac{1}{|A_n|+|B_n|+|C_n|} \bigg ( \int_{\big(\{u_0>0 \} \cup \{g>\Lambda[u_0] \} \big)_{-\eta} }g\;dS+ \Lambda[u_0]|B_n| \bigg)  \label{calc1}
\end{align}
Since by definition $A_n \subset \big(\{u_0>0\} \cup \{g>\Lambda[u_0] \}  \big)$, it follows that
\begin{equation}
  |A_n| \leq \big |  \{u_0>0\} \cup \{g>\Lambda[u_0] \}  \big |\;.
  \label{b1'}
\end{equation}
Moreover, \eqref{b2} and Lemma \ref{pm delta union} yield that for all $n\geq n^*$
\begin{align*}
C_n &\subset \big( ( \{u_0>0 \})_{+2\eta}  \cap  \{u_0=0\} \big ) \cup \big(  (\{g \geq \Lambda[u_0] \} )_{+2\eta}   \cap   \{g<\Lambda[u_0] \}  \big) \\
&=\big( ( \{u_0>0 \})_{+2\eta}  \setminus  \{u_0>0\} \big ) \cup \big(  (\{g \geq \Lambda[u_0] \} )_{+2\eta}   \setminus   \{g\geq \Lambda[u_0] \}  \big)\;.
\end{align*}
This in turn implies, by \eqref{eq:nondegeneracy2a} and \eqref{eq:B2-closed} that
\begin{align}
  |C_n| &\leq   \big| ( \{u_0>0 \})_{+2\eta}  \setminus  \{u_0>0\} \big | + \big|  (\{g \geq \Lambda[u_0] \} )_{+2\eta}   \setminus   \{g\geq \Lambda[u_0] \}  \big| \leq \omega_1(\eta) \; \label{b2'}
\end{align}
for some $\omega_1(\eta) \to 0$ as $\eta \to 0$.

\medskip
By \eqref{b1'} and \eqref{b2'} we deduce that for all $n\geq n^*$
\begin{equation}
  \frac{1}{|A_n|+|B_n|+|C_n|} \geq \frac{1}{\big |  \{u_0>0\} \cup \{g>\Lambda[u_0] \}  \big |+|B_n|+\omega_1(\eta)}\;.
  \label{subcalc1}
\end{equation}
In addition, the expression inside the parentheses on the right-hand side of \eqref{calc1} can be estimated below by
\begin{equation}\label{subcalc2}
  \bigg ( \dots \bigg )\geq \int_{\{u_0>0 \} \cup \{g>\Lambda[u_0] \} }g\;dS + \Lambda[u_0]|B_n| -(\max g) |D_\eta|\;,
\end{equation}
where $D_\eta:=\big(\{u_0>0 \} \cup \{g>\Lambda[u_0] \} \big)\setminus \big(\{u_0>0 \} \cup \{g>\Lambda[u_0] \} \big)_{-\eta}$.
We observe that by \eqref{eq:B2-open}
\begin{equation}\label{step}
  |D_\eta| \leq \omega_2(\eta)
\end{equation}
for some modulus of continuity $\omega_2$.

\medskip
We can plug now \eqref{subcalc1}, \eqref{subcalc2} and \eqref{step}  into \eqref{calc1} and obtain
\begin{align*}
  \lambda_n(t) &\geq  \frac{1}{\big |  \{u_0>0\} \cup \{g>\Lambda[u_0] \}  \big |+|B_n|+\omega_1(\eta)}  \bigg (   \int_{\big(\{u_0>0 \} \cup \{g>\Lambda[u_0] \} \big)\cup B_n}g\;dS -\max g\;\omega_2 (\eta)  \bigg)\;.
\end{align*}

By \eqref{initialdata}, there is a positive constant $M>0$ such that
\begin{equation*}
  \big |  \{u_0>0\} \cup \{g>\Lambda[u_0] \}  \big |+|B_n| \geq M>0.
\end{equation*}
Then we deduce that
\begin{align}\label{1stepAway}
  \lambda_{n}(t) \geq \fint_{\big(\{u_0>0 \} \cup \{g>\Lambda[u_0] \} \big)\cup B_n} g\;dS-\hat \omega(\eta) \;,
\end{align}
for some modulus of continuity $\hat \omega$.
We set $\tilde B_n:=\big(\{u_0>0 \} \cup \{g>\Lambda[u_0] \} \big)\cup B_n$ and we write $A_*^0=\big(A_*^0 \setminus \tilde B_n\big) \cup \tilde B_n$.
Due to \eqref{jumpset defn1}, we further observe that $P_n:=A_*^0 \setminus \tilde B_n \subset \{g=\Lambda[u_0] \}$.
This in turn implies that
\begin{align*}
  \fint_{A_*^0} g\;dS &=\frac{1}{|\tilde B_n|+|P_n|}\bigg( |\tilde B_n|\fint_{\tilde B_n} g\;dS+|P_n|\fint_{P_n} g\;dS \bigg)\\
  &=\frac{1}{|\tilde B_n|+|P_n|}\bigg( |\tilde B_n|\fint_{\tilde B_n} g\;dS+|P_n|\Lambda[u_0] \bigg).
\end{align*}
Since the maximum in \eqref{biglambda defn} is attained by the set $A^0_*$, it holds that
\begin{align*}
  (|\tilde B_n|+|P_n|)\Lambda[u_0]=|\tilde B_n|\fint_{\tilde B_n} g\;dS+|P_n|\Lambda[u_0]
\end{align*}
and $\fint_{\tilde B_n} g\;dS=\Lambda[u_0]$.
Therefore, we conclude by \eqref{1stepAway} that
\begin{align*}
 \lambda_{n}(t) \geq  \Lambda[u_0]-\hat \omega (\eta)\;.
\end{align*}

\medskip
{\bf Step 6:} To complete this proof, it remains to show \eqref{newpositivityset} and \eqref{initial jump}.

By \eqref{eq:un-monoton} we have $u_n\geq u$ for all $n$.
Due to \eqref{inclusions reg'}, setting $t_5(\eta)=t_4(\eta/2)$ this immediately yields the right inclusion in \eqref{newpositivityset}, that is
\begin{equation*}
  \{u(\cdot,t)>0\} \subset \{u_{n}(\cdot,t)>0\}  \subset \big(\{u_0>0 \} \cup \{g \geq \Lambda[u_0] \} \big)_{+\eta},
\end{equation*}
for all $t\in [0,t_5(\eta)]$.

In order to obtain the left inclusion, we deduce from \eqref{eq:conv-un} and \eqref{lower est noeta1}, \eqref{eq:step2a} that for all $t\in [0,t_4(\eta)]$ we have
\begin{align*}
  \Big(\{u_0>0\} \cup \{g>\Lambda[u_0]\}\Big)_{-\eta}
  &\subset \Big\{u_0>\frac{\delta(\eta)}{2}\Big\}\cup  \Big\{g>\Lambda[u_0]+\frac{\delta(\eta)}{2}\Big\}\\
  &\subset \Big\{u_n(\cdot,t)>\frac{\delta(\eta)}{4}\Big\}\cup  \Big\{u_n(\cdot,t)>\omega_2^*\Big(\frac{\delta(\eta)}{2},t\Big)\Big\}.
\end{align*}
Moreover, using \eqref{eq:conv-un} and passing with $n\to\infty$ to the limit gives
\begin{equation}
  \Big(\{u_0>0\} \cup \{g>\Lambda[u_0]\}\Big)_{-\eta}
  \subset \Big\{u(\cdot,t)\geq \min\Big\{\frac{\delta(\eta)}{4},\omega_2^*\Big(\frac{\delta(\eta)}{2},t\Big)\Big\}\Big\}
  \label{eq:inclusion-1}
\end{equation}
for all $t\in [0,t_4(\eta)]$. This justifies the left inclusion in \eqref{newpositivityset}.

\bigskip

For proving \eqref{initial jump}, we argue as in the Step $1$ and Step $5$ in order to obtain upper and lower bounds for $\lambda(t)-\Lambda[u_0]$ in terms of \eqref{newpositivityset}.
More precisely, for the upper bound we deduce by the left inclusion in \eqref{newpositivityset} and Lemma \ref{pm delta union} that
\begin{equation*}
  \{u_0>0\}_{-\eta} \subset \{u_0>0\}_{-\eta} \cup \{g>\Lambda[u_0]\}_{-\eta} \subset \{u(\cdot,t)>0\} \quad \text{for all }\; t\in [0,\bar t(\eta)]\;.
\end{equation*}
Similar to Step $1$, we deduce using the first item of Proposition \ref{Lambda_properties1} and \eqref{eq:lambdas-ineq} in  Corollary \ref{Lambda properties2} that
\begin{equation*}
  \Lambda \big [\{u_0>0 \}_{-\eta} \big ] \geq \Lambda \big [\{u(\cdot,t)>0 \} \big ] \geq \lambda(t).
\end{equation*}
Then, the second item of Proposition \ref{Lambda_properties1} and Lemma \ref{-delta set} yield the existence of a modulus of continuity $\omega_3$ such that $\lambda(t)-\Lambda[u_0]\leq \omega_3(\eta)$ holds for all $t\in [0,\bar t(\eta)]$.
After possibly enlarging $\eta$ and redefining $\bar t(\eta)$ we deduce
\begin{equation*}
  \lambda(t)-\Lambda[u_0] \leq \eta\, \quad \text{for all }0 < t \leq \bar t(\eta)\; .
\end{equation*}

For the lower bound, we rewrite $\{u(\cdot,t)>0\}$ as the following union of disjoint sets
\begin{equation*}
  \{u(\cdot,t)>0\} = A \cup B \cup C
\end{equation*}
with
\begin{align*}
  A &:= \{u(\cdot,t)>0\}\cap \big(  \{u_0>0\} \cup \{g>\Lambda[u_0] \}  \big),\\
  B &:= \{u(\cdot,t)>0\}\cap \big(  \{u_0=0\} \cap \{g=\Lambda[u_0] \}  \big),\\
  C &:= \{u(\cdot,t)>0\}\cap \big(  \{u_0=0\} \cap \{g<\Lambda[u_0] \}  \big).
\end{align*}
Following the same line of arguments as in Step $5$, we infer by means of \eqref{newpositivityset} the existence of a modulus of continuity $\omega_4$
\begin{equation*}
  \lambda(t)-\Lambda[u_0] \geq -\omega_4(\eta)\, \quad \text{for all }0 < t \leq \bar t(\eta)\; .
\end{equation*}
Again, possibly enlarging $\eta$ and redefining $\bar t(\eta)$ we deduce $\lambda(t)-\Lambda[u_0] \geq -\eta$.
This completes the proof of \eqref{initial jump}.
\end{proof}

We finally give the proof of the simplified version Theorem \ref{1thm:main2} of our results stated in the introduction.

\begin{proof}[Proof of Theorem \ref{1thm:main2}]
The convergence of $\lambda(t)$ with $t\searrow 0$ follows from \eqref{initial jump}.

By \eqref{eq:nondegeneracy1} and the additional assumption $|\{g=\Lambda[u_0]\}|=0$ we obtain from \eqref{newpositivityset} that
\begin{equation*}
  (A_*^0)_{-2\eta} \subset \{u(\cdot,t)>0\} \subset (A_*^0)_{+2\eta}
\end{equation*}
for all $0<t<\bar{t}(\eta)$.
From this property and from \eqref{eq:inclusion-1}, which is analogue to \eqref{inclusions2}, we deduce the convergence of supports exactly as in the proof of Theorem \ref{1thm:main1}.
Together with Corollary \ref{cor:expl2} this proves Theorem \ref{1thm:main2}.
\end{proof}

%\section{Outlook: Initial jump for degenerate data if the second nondegeneracy condition is violated}\label{sec:jump2}
%
%We stress that both \eqref{eq:nondegeneracy1-lambda} and \eqref{eq:nondegeneracy2} are necessary in order to obtain any continuity properties for the function $\lambda$.
%In fact, we will present in \cite{LNRVloading} an example of initial data $u_0$ for which \eqref{eq:nondegeneracy1} holds while \eqref{eq:nondegeneracy2} is invalid.
%We will prove that under these assumptions, the function $\lambda$ is not continuous at $t=0$ and the positivity set $\{u(\cdot,t)>0\}$ is oscillatory as $t\to 0^+$.
%
%Since the analysis is rather involved and lengthy at this point we will only sketch the main ideas and the key results.
%

\section{Comparison with the classical parabolic obstacle problem}
\label{sec:classical}

We collect here some analogous continuity and jump properties for the interfaces associated to the classical parabolic free boundary problem \eqref{eq:obs-class1}-\eqref{eq:obs-class3}.
\bigskip

We define $H:\R\to\{0,1\}$, $H=\Chi_{(0,\infty)}$ as the characteristic function of the positive real numbers. Let $\mathcal U$ be a smooth compact manifold without boundary embedded in $\mathbb R^n$ and also let $T>0$. Then we can reformulate problem \eqref{eq:obs-class1}-\eqref{eq:obs-class3} as follows
\begin{equation}\label{fbpgR}
\begin{cases}
  &\partial_t u -\Delta u  =fH(u)\quad
	\text{ on } \mathcal U \times(0,T) \,
	\\
	&f \leq 0 \qquad \qquad \qquad \;\; \;\text{ a.e. in } \{u=0\}
	\\
	&u(\cdot,0)=u_0 \qquad \quad \;\;\;\;\; \text{  on } \mathcal U \,
\end{cases}
\end{equation}
where $f \in C( \mathcal {U}\times[0,T])$.

Clearly the only difference between this problem and the problem under consideration \eqref{eq:oblam1}-\eqref{eq:oblam4} (see also the equivalent formulation \eqref{eq:oblamalt2}-\eqref{eq:oblamalt4}), is the absence of the nonlocal term $\lambda$ which depends on the positivity set $\{u(\cdot,t)>0\}$.
\medskip

The well-posedness as well as the properties of the interfaces of \eqref{fbpgR} has been studied in detail over the past years.
For the continuity of the corresponding positivity set $\{u(\cdot,t)>0\}$, we state the following theorems.

\begin{theorem}\label{StefanCont}
Assume that $f \in  C(\mathcal U\times[0,T])$ for some $T>0$ and consider
nonnegative initial data $u_0 \in C(\mathcal U)$.
If in addition we assume that $f \leq -\theta$ for some fixed $\theta>0$ in a neighborhood of $\{u_0=0 \} \times [0,T]$, then the set $\{u(\cdot, t)>0\}$ is continuous at $t=0$ in the sense that for any $\eta>0$, there exists $t^*(\eta)>0$ such that
\begin{equation*}
\{u_0>0 \}_{-\eta} \subset \{u(\cdot,t)>0\} \subset \{u_0>0 \}_{+\eta}
\end{equation*}
for all $0 \leq t \leq t^*(\eta)$.
\end{theorem}

One can prove Theorem \ref{StefanCont} adapting the arguments in the proofs of \cite[Theorem 4.2]{BF76} and \cite[Theorem 3.2 (ii)]{EK79} .

\begin{theorem}\label{StefanJump}
Suppose that $u$ is a solution to \eqref{fbpgR}  with initial data $u_0 \in C(\mathcal U)$.
The positivity set $\{u(\cdot,t)>0\}$ converges to the set $\{u_0>0\} \cup \{ f(\cdot,0)>0\}$ as $t \to 0^{+}$ in the sense that  for any $\eta>0$, there exists $t^*(\eta)>0$ such that
\begin{equation*}
\big (\{u_0>0 \}\cup \{ f(\cdot,0)>0\}\big)_{-\eta} \subset \{u(\cdot,t)>0\} \subset \big(\{u_0>0 \}\cup \{ f(\cdot,0)\geq 0\}\big)_{+\eta}
\end{equation*}
for all $0 \leq t \leq t^*(\eta)$.
\end{theorem}

The proof of Theorem \ref{StefanJump} can be obtained arguing in a similar way as in the proof of Theorem \ref{main.p3}. Now the proof will follow by simpler techniques, since the dependence of the non local term $\alpha$ on the positivity set $\{u(\cdot,t)>0\}$ does not occur here.

\begin{remark_nn}
In particular, we notice that in the case of problem \eqref{fbpgR} the positivity set $\{ u(\cdot,t)>0\}$ as $t\to 0^+$ depends only on the set $\{u_0>0\}$ and the positivity set of $f(\cdot,t)$.
The strict inequality $f<0$ in a neighborhood of $\{u_0=0\}$ plays a similar role to the condition \eqref{eq:nondegeneracy1-lambda} for the problem \eqref{eq:oblam1}-\eqref{eq:oblam4}.
However, since $f(\cdot,t)$ in \eqref{fbpgR} does not depend on the positivity set $\{u(\cdot,t)>0\}$, we can omit \eqref{eq:nondegeneracy2} here.
We recall that imposing such a regularity condition  for the set $\{u_0>0\}$, was necessary in order to prove continuity properties for the function $\lambda$ (or equivalently $\alpha$ in \eqref{eq:obalph1}-\eqref{eq:obalph4}) that depends on the support of the solution $u$.
\end{remark_nn}

\begin{appendices}
\section{ On the $\pm \delta$-sets} \label{A1}

Throughout this section let $\Gamma \subset \mathbb R^3$ be a smooth compact surface without boundary.
In the formulation of the nondegeneracy condition for the initial data the following definition of inner and outer approximations where used.

\begin{definition} \label{app delta sets defn}
	For a Borel set $A\subset \Gamma$ and $\delta>0$ we set
	\begin{equation}
	A_{+\delta}:=\{x\;| d(x,A) \leq \delta \}\;, \quad A_{-\delta}:=\{x\;| d(x,A^{\mathsf{c}}) \geq \delta \}\;,
	\label{delta sets}
	\end{equation}
  where for a set $E\subset \Gamma$ the function $d(\cdot,E)$ denotes the associated distance function on $\Gamma$ induced by the Euclidean distance in the ambient space.
  In particular, $A_{\pm \delta}$ are subsets of $\Gamma$.
\end{definition}

% \begin{remark}\label{referee-req}
% We stress that $A_{\pm \delta}$ are always going to be understood as subsets of $\Gamma$.
% \end{remark}

First, we are going to justify some properties of the sets $A_{\pm \delta}$, which will be useful in our analysis.

\begin{lemma} \label{delta-sets-complement}
	\begin{enumerate}
		\item (Monotonicity) For any Borel set $A\subset\Gamma$ and $0<\delta\leq r$ we have
		\begin{equation*}
		A\subset A_{+\delta}\subset A_r \quad\text{ and }\quad
		A_{-r}\subset A_{-\delta}\subset A.
		\end{equation*}
		\item (Complements) For all $\delta>0$, the following inclusions hold true
		\begin{align}
		\cmp{(A_{+\delta})} &= (\cmp{A})_{-\delta-0} := \bigcup_{r>\delta}(\cmp{A})_{-r} \subset (\cmp{A})_{-\delta}\label{eq:cmp-1}\\
		\cmp{(A_{-\delta})} &= (\cmp{A})_{+\delta-0} := \bigcup_{r<\delta}(\cmp{A})_{+r} \subset (\cmp{A})_{+\delta}.\label{eq:cmp-2}
		\end{align}
		In particular, $(\cmp{A})_{+\frac{\delta}{2}}\subset \cmp{(A_{-\delta})}\subset (\cmp{A})_{+\delta}$ and $(\cmp{A})_{-2\delta}\subset \cmp{(A_{+\delta})}\subset (\cmp{A})_{-\delta}$.
	\end{enumerate}
\end{lemma}

\begin{proof}
	\begin{enumerate}
		\item The monotonicity follows immediately using \eqref{delta sets}.
		\item Due to \eqref{delta sets}, we can write
		\begin{align*}
		(A_{+\delta})^{\mathsf{c}}=\big (\{x\;| d(x,A) \leq \delta \} \big )^{\mathsf{c}}= \{x\;| d(x,A) > \delta \} = \bigcup_{r>\delta} \{x\;| d(x,A) \geq r\} =\bigcup_{r>\delta} (\cmp{A})_{-r}
		\end{align*}
and by means of the monotonicity obtained in the first item, we deduce that
		\begin{equation*}
		(\cmp{A})_{-r}\subset (\cmp{A})_{-\delta}\;, \quad \text{ for all } \;r>\delta\;.
		\end{equation*}
This proves the first inclusion.
Arguing in a similar way, we prove also the second inclusion.
	\end{enumerate}

\end{proof}

The significance of the sets $A_{\pm \delta}\;$, as defined in \eqref{delta sets}, becomes evident in the following lemmas.

\begin{lemma}\label{pm delta union}
	For any $\delta>0$
	\begin{align*}
	\big (A_{-\delta} \cup C_{-\delta} \big) \subset \big (A\cup C \big)_{-\delta}&\subset \big ( A_{-\delta} \cup C_{-\delta} \big) \cup \big (  A_{+\delta} \setminus A_{-\delta}  \big)\,\\
	\big (A\cup C \big)_{+\delta}&= \big ( A_{+\delta} \cup C_{+\delta} \big).
	\end{align*}
\end{lemma}

\begin{proof}

	By definition, it follows easily that $\big (A_{-\delta} \cup C_{-\delta} \big) \subset \big (A\cup C \big)_{-\delta}$.
Hence, we write $\big (A\cup C \big)_{-\delta}=\big (A_{-\delta} \cup C_{-\delta} \big) \cup \big [ \big (A\cup C \big)_{-\delta} \setminus \big (A_{-\delta} \cup C_{-\delta} \big) \big]$ and we fix $x \in \big (A\cup C \big)_{-\delta}\;$.
We observe that there exists a ball $B_{\delta}(x) \subset A\cup C$ such that $B_{\delta}(x) \cap A \neq \emptyset $ and $B_{\delta}(x) \cap C \neq \emptyset \;$.
This in turn implies that $d(x,A) \leq \delta$ and in particular that $x\in A_{+\delta}\;$.
Therefore,
	\begin{align*}
	\big (A\cup C \big)_{-\delta} \setminus \big (A_{-\delta} \cup C_{-\delta} \big) \subset A_{+\delta} \setminus \big (A_{-\delta} \cup C_{-\delta} \big) \subset A_{+\delta} \setminus A_{-\delta} \;.
	\end{align*}
Moreover, for any $x \in A_{+\delta} \cup C_{+\delta}$ we have
	\begin{align*}
	x \in  A_{+\delta} \cup C_{+\delta} & \Leftrightarrow d(x,A) \leq \delta \quad \text{or} \quad d(x,C) \leq \delta  \\
	& \Leftrightarrow d(x,A\cup C) \leq \delta \\&\Leftrightarrow x\in \big (A\cup C \big)_{+\delta}\;.
	\end{align*}
\end{proof}

\begin{lemma}\label{-delta set}
	Let $h\in C(\Gamma)$ be given with $\{h>0\}$ and $\{h\leq 0\}$ both being non-empty.
	Then there exist positive non-decreasing functions $\delta_1,\delta_2:(0,1)\to\R$ with $\lim \limits_{r \to 0} \delta_1(r)=\lim \limits_{r \to 0} \delta_2(r)=0$ such that
	\begin{equation*}
	(\{ h>0 \})_{-\delta_1(r)} \subset \{h \geq r \} \subset (\{ h>0 \})_{-\delta_2(r)}\;.
	\end{equation*}
\end{lemma}

\begin{proof}
	We first prove the second inclusion.
We observe that for any $r>0$ the sets $\{ h \geq r\}$ and $\{h \leq 0 \}$ are compact and disjoint.
	Thus, setting
	\begin{equation*}
	\delta_2(r):=d\big(\{ h \geq r\},\{h \leq 0 \}\big)
	\end{equation*}
	yields a positive non-decreasing function.
	Moreover, by definition of $\delta_2$, in $\{ h \geq r\}$ we have $d(\cdot,\{h\leq 0\})\geq \delta_2(r)$, which implies $\{ h \geq r\}\subset (\{h> 0\})_{-\delta_2(r)}$.

	Since $\delta_2$ is non-decreasing $\omega:=\lim_{r\searrow 0}\delta_2(r)$ exists.
Assume $\omega>0$.
Then
	\begin{equation*}
	d\big(\cdot,\{h\leq 0\}\big)\geq \omega\quad\text{ in }\{h>0\}=\bigcup_{r>0}\{ h \geq r\},
	\end{equation*}
	which yields a contradiction since $h$ is continuous and $\partial\{h> 0\}\subset\{h\leq 0\}$ is non-empty.

	\bigskip
	To prove the first inclusion define for $\delta>0$ the function
	\begin{equation*}
	m(\delta) := \min \big\{h(x)\,:\, d(x,\{h\leq 0\})\geq\delta\big\}.
	\end{equation*}
	We observe that $m:(0,1)\to\R$ is well-defined, positive and non-decreasing.

	By definition we have
	\begin{equation}
	(\{h>0\})_{-\delta}\subset \{h\geq m(\delta)\}.
	\label{eqA:1001}
	\end{equation}

	Assume that $\omega:=\lim_{\delta\searrow 0}m(\delta)>0$.
Then by continuity of $h$
	\begin{equation*}
	h\geq \omega\quad\text{ in }\bigcup_{\delta>0}(\{h>0\})_{-\delta}=\big\{d(\cdot,\{h\leq 0\})>0\big\}=\{h>0\},
	\end{equation*}
	a contradiction.

	We next define
	\begin{equation}\label{intermed}
	\bar \delta_1(r) := \inf\{\delta>0\,:\, m(\delta)\geq r\}
	\end{equation}
and we notice that $\bar \delta_1$ is positive, non-decreasing and moreover $\lim \limits_{r \to 0} \bar \delta_1(r)=0$ since $\lim_{\delta\searrow 0}m(\delta)=0$.
Hence, if we set $\delta_1(r):=\bar \delta_1(r)+r$, we deduce by \eqref{intermed} and the monotonicity of $m$ that

	\begin{equation*}
	m(\delta_1(r)) \geq r\;,
	\end{equation*}
for $\delta_1(r) \to 0$ as $r\to 0$.
	Then we conclude from \eqref{eqA:1001} that
	\begin{equation*}
	(\{h>0\})_{-\delta_1(r)}\subset \{h\geq m(\delta_1(r))\}\subset \{h\geq r\}.
	\end{equation*}

\end{proof}

\section{On the regularity of sets}\label{A2}

\begin{definition}\label{app regularity defn}
Assume that $\Gamma \subset \mathbb R^3$ is a smooth surface.
We will call a Borel set $A \subset \Gamma$ regular if
\begin{equation}\label{newremedy'}
\vert \partial A \vert=0\;,
\end{equation}
where $\vert \cdot \vert=\mathcal H^2$ denotes the two-dimensional Hausdorff measure.
\end{definition}

To begin with, we prove an equivalent characterization for \eqref{newremedy'}.

\begin{lemma}\label{non-fat}
For some Borel set $A\subset \Gamma$ and $\delta>0$, we consider the sets $A_{\pm \delta}$ as these are given by \eqref{delta sets}.
Then, the set $A$ is regular iff
\begin{equation*}
  |A_{+\delta} \setminus A_{-\delta}| \to 0\;\quad \text{as } \delta \to 0 \;.
\end{equation*}
Moreover, we have that
\begin{align}
  \lim_{\delta\searrow 0}|A_{+\delta} \setminus A| &\searrow 0 \quad\text{ for all }A\subset\Gamma\text{ closed,}
  \label{eq:B2-closed}\\
  \lim_{\delta\searrow 0}|A \setminus A_{-\delta}| &\searrow 0 \quad\text{ for all }A\subset\Gamma\text{ open.}
  \label{eq:B2-open}
\end{align}
\end{lemma}

\begin{proof}
By the monotonicity properties stated in Lemma \ref{delta-sets-complement}, for $\delta_1>\delta_2$ it holds that $A_{+\delta}$ is decreasing and $A_{-\delta}$ is increasing as $\delta \searrow 0$.
Moreover,
\begin{equation*}
  \bigcap \limits_{\delta>0}A_{+\delta}=\{x\;| d(x,A) =0 \} = \overline{A}
  \quad\text{ and }\quad
  \bigcup \limits_{\delta>0}A_{-\delta}=\{x\;| d(x,A^c) >0 \} = \mathring{A}.
\end{equation*}
This implies
\begin{align}\label{fact1a}
  \lim \limits_{\delta \searrow 0} \vert A_{+\delta} \setminus A \vert
  &=\big \vert \bigcap \limits_{\delta>0}A_{+\delta} \setminus A\big \vert
  = |\overline{A}\setminus A|\;,\\
  \label{fact1b}
  \lim \limits_{\delta \searrow 0} \vert A\setminus A_{-\delta} \vert
  &=\big \vert \bigcap \limits_{\delta>0}A\setminus A_{-\delta}\big \vert
  = |A\setminus\mathring{A}|\;,
\end{align}
and
\begin{equation*}
  \lim \limits_{\delta \searrow 0} |A_{+\delta} \setminus A_{-\delta}|
  =|\bigcap \limits_{\delta>0}\big(A_{+\delta} \setminus A_{-\delta}\big)|
  =|\overline{A}\setminus\mathring{A}|=|\partial A|\;.
\end{equation*}
Therefore, if the set $A$ is regular according to Definition \ref{app regularity defn}, then $\lim \limits_{\delta \searrow 0} \vert A_{+\delta} \setminus A_{-\delta} \vert =0$ and vice versa.

The other claims follow by \eqref{fact1a}, \eqref{fact1b}.
\end{proof}

Next, we will show that the union of two regular sets is also regular.

\begin{lemma}\label{delta union regular}
Suppose that both $A$ and $C$ are regular.
Then, the union $A\cup C$ is  also regular in terms of Definition \ref{app regularity defn}.
\end{lemma}

\begin{proof}
It holds that $\partial  (A \cup C) \subset \partial A \cup \partial C$.
This in particular implies that
\begin{equation*}
\vert \partial  (A \cup C) \vert \leq \vert \partial A \cup \partial C \vert \leq \vert \partial A \vert + \vert \partial C \vert \;.
\end{equation*}
Using Definition \ref{app regularity defn} the claim follows immediately.

\end{proof}

\section{Construction of suitable initial data}\label{A:initial_data}
In Lemma \ref{eta data regular} we have claimed that we can construct initial data that enjoy suitable convergence, positivity and monotonicity properties.
In this section, we provide a proof of the lemma.

\begin{proof}[Proof of Lemma \ref{eta data regular}]
Let $(\gamma_n)_n$ be a positive, non increasing sequence with $2\gamma_{n+1}<\gamma_n$ for all $n\in \mathbb N$.
In particular, $\gamma_n\searrow 0$ as $n\uparrow\infty$.
We are going to construct $u^0_n$ as follows.

\medskip
We write $\Gamma=\{ g \leq  \Lambda[u_0]-2{\gamma_n}\} \cup \{\Lambda[u_0]-2{\gamma_n} < g < \Lambda[u_0]-{\gamma_n} \} \cup \{ g\geq \Lambda[u_0]-{\gamma_n}\}$.
Since $\Gamma$ is smooth and $g\in C^2(\Gamma)$, by the Morse-Sard Theorem \cite[Theorem 3.1.3]{Hirs76} almost all $r\in g(\Gamma)$ are regular values of $g$.
Moreover, for all regular values $r$ the sets $\{g=r\}$ are one-dimensional $C^2$-submanifolds of $\Gamma$.
In particular, for such values the sets $\{g>r\}$ are regular in the sense of  Definition \ref{app regularity defn}.

Hence, for all $n\in\N$ we can fix a regular value
\begin{equation}
  r_n \in [\Lambda[u_0]-2{\gamma_n},\Lambda[u_0]-{\gamma_n}].
  \label{eq:rn}
\end{equation}
This choice and the fact that $2\gamma_{n+1}<\gamma_n$ for all $n\in \mathbb N$ imply that $(r_n)_n$ is a strictly monotone increasing sequence.

\medskip
Since $g$ is continuous and $(r_n)_n$ is strictly monotone increasing, the sets $\{g>r_n\}$ are open and satisfy $\{g>r_{n+1}\}\ssubset \{g>r_n\}$.
Therefore, applying Lemma \ref{ind claim} below, we obtain that for all $n\in\N$ there exists $\zeta_n \in C^{\infty}(\Gamma)$ such that
\begin{equation}\label{zeta defn}
  \zeta_n=1\;\; \text{in } \{g>r_{n+1}\}, \quad \zeta_n \in (0,1] \;\;\text{in } \{r_n <g \leq r_{n+1} \}, \quad \zeta_n=0 \;\; \text{in } \{g\leq r_n\}.
\end{equation}
At this point, we define $u^0_n$ as
\begin{equation}\label{eta data construct}
  u^0_n=u_0+{\gamma_n} \zeta_n\;.
\end{equation}

By \eqref{eta data construct} it follows immediately that $u^0_n \geq u_0$ for all $n$.
Since $\gamma_n\searrow 0$ and by \eqref{zeta defn}, we have that $\gamma_{n+1} \leq \gamma_n$ and further that $\zeta_{n+1} \leq 1=\zeta_n$ in $\{g>r_{n+1}\}$.
Hence, we obtain that $u^0_{n+1} \leq u^0_n$ and item \ref{first} is justified. Item \eqref{nine} is a direct consequence of \eqref{eta data construct}, $0\leq\zeta_n\leq 1$ and $\gamma_n\searrow 0$.

We easily observe that item \eqref{second} follows from item \eqref{first}, and that item \eqref{third} and \eqref{fourth} in Lemma \ref{eta data regular} follow by \eqref{zeta defn} and \eqref{eta data construct}.
Moreover since $\{ u^0_n>0\}=\{ u_0>0\}\cup \{ \zeta_n>0\}$, we derive \eqref{sixth} by means of Lemma \ref{delta union regular}.
In order to deduce \eqref{fifth}, we combine \eqref{fourth} with Lebesgue's dominated convergence theorem.
Then
\begin{equation*}
 \vert \{u^0_n>0 \} \cap \{u_0=0 \} \cap \{g<\Lambda[u_0] \} \vert \leq \vert \{\Lambda[u_0]-2{\gamma_n}<g<\Lambda[u_0] \} \vert \to 0
\end{equation*}
as $n \to \infty$.

\bigskip

At this point we prove item \ref{eighth}.
We recall due to \eqref{eq:lambdas-ineq} and the fact that $\{u_0>0\} \subset \{u^0_n>0 \}$ that $\lambda^0_n \leq \Lambda[u_0]$.
In particular we have that
\begin{align}\label{calcStarnew}
  \lambda^0_n=\frac{1}{\vert A^0_* \vert +\vert  \{u_0=0\} \cap \{r_n<g<\Lambda[u_0] \}\vert}\bigg(\int \limits_{A^0_*}g\;dS+\int \limits_{\{u_0=0\}\cap \{r_n<g<\Lambda[u_0] \}} g\;dS \bigg).
\end{align}
For the sake of simplicity we set $B:=\{u_0=0\}\cap \{r_n<g<\Lambda[u_0] \}$.
By means of Lemma \ref{well-posed}, \eqref{calcStarnew} implies
\begin{align}\label{calcStarfinal}
  \lambda^0_n&=\frac{1}{\vert A^0_* \vert+\vert B \vert} \bigg(\Lambda[u_0]\vert A^0_* \vert +\int \limits_{\{u_0=0\}\cap \{r_n<g<\Lambda[u_0] \}} g\;dS\bigg) \notag \\&=\frac{1}{\vert A^0_* \vert+\vert B \vert}\bigg(\Lambda[u_0]\vert A^0_* \vert +\Lambda[u_0]\vert B \vert +\int \limits_{B} \big(g-\Lambda[u_0] \big)\;dS \bigg) \notag \\&=\Lambda[u_0]+\frac{1}{\vert A^0_* \vert+\vert B \vert}\int \limits_{B} \big(g-\Lambda[u_0] \big)\;dS\;.
\end{align}
Employing once again \eqref{eq:lambdas-ineq} and the fact that $r_n \in [\Lambda[u_0]-2{\gamma_n},\Lambda[u_0]-{\gamma_n}]$, we obtain that
\begin{equation*}
  0\leq \Lambda[u_0] -\lambda^0_n <\frac{2{\gamma_n} \vert B \vert}{\vert A^0_*\vert+\vert B \vert}\;.
\end{equation*}
Furthermore, since $B=\{u^0_n>0\}\cap \{u_0=0\}\cap \{g<\Lambda[u_0] \}$ we deduce by item \ref{fifth} and $|A_*^0|\geq |\{u_0>0\}|>0$ that $\vert B \vert \to 0 $ as $n\to \infty$.
This in particular yields that there exists $n^*\in\N$ such that
\begin{equation*}
\vert \lambda^0_n-\Lambda[u_0] \vert < \frac{{\gamma_n}}{4}, \quad \text{for all }n> n^*\;.
\end{equation*}
We can assume that $n^*=1$, otherwise we pass to the sequences $(\gamma_{n+n^*})_{n\in\N}$, $(u_{n+n^*}^0)_{n\in\N}$.
Therefore, the proof of item \ref{eighth} is complete.

\medskip
It remains to prove item \ref{seventh} of the lemma.
First we show the right-hand side inclusion in item.
To this end, we notice that due to item \ref{fourth} it holds
\begin{equation*}
  \{u^0_n>0 \} \subset  \{g>\Lambda[u_0]-2{\gamma_n}\} \cup \{u_0>0\}.
\end{equation*}
Furthermore, for any $\eta>0$ we have that $\{u_0>0 \} \subset \big (\{u_0>0\}\big )_{+\eta}$.
We claim that it suffices to show that there exists  $n^*=n^*(\eta)$ with $n^*(\eta) \to 0$ as $\eta \to 0$ such that
\begin{equation}\label{ess calc}
  \{g>\Lambda[u_0]-2\gamma_{n^*}\} \subset \big(\{g \geq \Lambda[u_0] \} \big)_{+\eta}\;.
\end{equation}
Indeed, if \eqref{ess calc} holds, then by virtue of Lemma \ref{pm delta union} and the monotonicity property in item \ref{second} we conclude that
\begin{align*}
  \{u^0_n>0 \}\subset\{u^0_{n^*}>0 \} \subset \big (\{u_0>0\}\big )_{+\eta} \cup \big(\{g \geq \Lambda[u_0] \} \big)_{+\eta}=\big (  \{u_0>0\} \cup \{g \geq \Lambda[u_0] \} \big)_{+\eta}\;
\end{align*}
for any $n \geq n^*$.
In order to show \eqref{ess calc}, we recall Lemma \ref{delta-sets-complement} and Lemma \ref{-delta set}.
More precisely, we compute
\begin{align*}
  \big(\{g \geq \Lambda[u_0] \} \big)_{+\eta}=\bigg( \big(\{g < \Lambda[u_0] \} \big)^{\mathsf{c}} \bigg)_{+\eta} \supset \bigg ( \big(\{g < \Lambda[u_0] \} \big)_{-\eta} \bigg )^\mathsf{c} \supset \bigg( \{g \leq \Lambda[u_0] -2\gamma_{n^*} \} \bigg)^\mathsf{c}
\end{align*}
for some $n^*(\eta)>0$ such that $n^*(\eta) \to \infty$ as $\eta \to 0$.

The proof is complete.

\end{proof}

We now prove the claim in \eqref{claim}.

\begin{lemma}\label{ind claim}
For any two open sets $U_1 \ssubset U_2\subset \Gamma$, there exists $\zeta \in C^{\infty}(\Gamma)$ such that
\begin{equation}\label{claim}
  \zeta=1 \,\text{ in }U_1,\quad
  \zeta>0 \,\text{ in }U_2, \quad
  \zeta=0 \,\text{ in }\Gamma \setminus U_2,\quad
  0\leq\zeta\leq 1\,\text{ in }\Gamma\;.
\end{equation}
\end{lemma}

\begin{proof}
In a first step we construct $\psi\in C^\infty(\Gamma)$ with $\psi>0$ in $U_2$ and $\psi=0$ in $\Gamma\setminus U_2$.

We choose a sequence $(x_j)_j$ in $U_2$ such that $\{x_j\,:\,j\in\N\}$ is dense in $U_2$ and set $\rho_j:=\frac{1}{2}d(x_j,\Gamma\setminus U_2)>0$.

Next, we fix a nonnegative function $\phi \in C^\infty(\R^3)$ that vanishes outside the unit ball $B_1(0)$ and satisfies $0\leq\phi\leq 1$ in $\R^3$ and $\phi>0$ in $B_1(0)$.
We define $\psi:\Gamma\to\R$ by
\begin{equation*}
  \psi(x):=\sum \limits_{j\in\N} 2^{-j} c_j\phi \bigg(\frac{x-x_j}{\rho_j} \bigg),
\end{equation*}
where $c_j>0$ is chosen such that
\begin{equation*}
  \big\|\partial^\alpha \phi \bigg(\frac{\cdot-x_j}{\rho_j} \bigg)\big\|_{C^0(\Gamma)} \leq \frac{1}{c_j}
\end{equation*}
for all (covariant) partial derivatives of order $|\alpha|\leq j$.

We observe that $0\leq \psi\leq 1$ is well-defined and smooth with $\psi=0$ outside $U_2$.
Next, we claim that $\psi>0$ in $U_2$.

In fact, for all $x\in U_2$, there exists $(x_{j(k)})_k$ with $x_{j(k)}\,\to\, x \,\; \text{ as }\; k\to\infty$.
This implies
\begin{equation*}
  \lim_{k\to\infty}\rho_{j(k)}= \frac{1}{2}\lim_{k\to\infty} d(x_{j(k)},\Gamma\setminus U_2)=\frac{1}{2} d(x,\Gamma\setminus U_2)>0.
\end{equation*}
Therefore $x\in B(x_{j(k)},\rho_{j(k)})$ for $k$ sufficiently large, and hence $\psi(x)>0$.

\medskip
Next, we choose $U_1\ssubset V\ssubset U_2$ and a bump function $\vartheta\in C^\infty(\Gamma)$, $0\leq \vartheta\leq 1$ with $\vartheta=0$ outside $V$ and $\vartheta=1$ in $U_1$.
Finally, we set
\begin{equation*}
  \zeta = (1-\vartheta)\psi + \vartheta
\end{equation*}
and observe that $\zeta=1$ in $U_1$, that $\zeta=\psi=0$ outside $U_2$ and that $\zeta\geq \psi>0$ in $U_2$, and that $\zeta\leq \max\{\psi,1\}=1$.
\end{proof}

\end{appendices}

%\bibliographystyle{plain}

%\bibliography{LNRV22}

\end{document}